\documentclass[12pt]{amsart}
\usepackage{amsmath,amsfonts,amssymb,amsthm,epsfig,amscd,epsfig,psfrag,comment,hyperref}
\usepackage{tikz}
\usepackage[all]{xy}
\usepackage{calligra,mathrsfs}

\numberwithin{equation}{section}
\newtheorem{Theorem}[equation]{Theorem}
\newtheorem{Proposition}[equation]{Proposition}
\newtheorem{Lemma}[equation]{Lemma}

\newtheorem{Corollary}[equation]{Corollary}

\theoremstyle{definition}
\newtheorem{Remark}[equation]{Remark}

\newtheorem{Definition}[equation]{Definition}

\numberwithin{figure}{section}

\newcommand{\Coh}{\mathrm{Coh}}
\newcommand{\QCoh}{\mathrm{QCoh}}

\def\Nilp{{\rm{Nilp}}}
\def\gr{{\rm{gr}}}
\def\GL{\rm{GL}}
\def\Hilb{{\rm Hilb}}
\def\Kar{{\rm Kar}}
\def\Tr{{\rm Tr}}
\def\bkappa{\overline{\kappa}}

\def\MF{{\rm{MF}}}
\def\Rep{{\rm{Rep}}}
\def\bO{{\bf{O}}}

\def\g{\mathfrak{g}}

\def\P{\mathbb{P}}
\def\ti{\widetilde{i}}
\def\tN{\widetilde{N}}
\def\cl{{\rm{cl}}}
\def\veps{\varepsilon}

\def\f{\mathbf{f}}
\def\ur{\underline{r}}
\def\uell{\underline{\ell}}
\def\pt{pt}
\def\tj{\widetilde{j}}
\def\htau{\widehat{\tau}}

\def\mfp{{\mathfrak{p}}}
\def\tpi{\widetilde{\pi}}
\def\r{{\bf{r}}}

\def\L{\Lambda}

\def\tm{\widetilde{m}}

\def\C{\mathbb{C}}

\def\N{\mathbb{N}}
\def\bN{\mathbb{N}}
\def\Z{\mathbb{Z}}

\def\Spec{{\mbox{Spec}}}

\def\id{{{id}}}

\def\tcF{{\widetilde{\mathcal{F}}}}

\def\O{{\mathcal O}}

\def\dim{\mbox{dim}}

\def\la{\langle}
\def\ra{\rangle}

\def\H{{\mathrm{H}}}
\def\bH{{\mathbb{H}}}

\def\cC{{\mathcal{C}}}

\def\cF{{\mathcal{F}}}
\def\cG{{\mathcal{G}}}
\def\cH{{\mathcal{H}}}

\def\cK{{\mathcal{K}}}

\def\cM{{\mathcal{M}}}
\def\cN{{\mathcal{N}}}
\def\cP{{\mathcal{P}}}

\def\cU{{\mathcal{U}}}
\def\cV{{\mathcal{V}}}

\def\cX{{\mathcal{X}}}
\def\cY{{\mathcal{Y}}}

\def\oQ{{\overline{Q}}}
\def\tQ{{\widetilde{Q}}}
\def\oE{{\overline{E}}}
\def\tE{{\widetilde{E}}}
\def\bA{{\mathbb{A}}}
\def\bM{{\mathbb{M}}}

\def\sl{{\mathfrak{sl}}}
\def\gl{{\mathfrak{gl}}}
\def\hG{{\widehat{G}}}
\def\hP{{\widehat{P}}}

\DeclareMathOperator{\Hom}{Hom}

\DeclareMathOperator{\End}{End}
\DeclareMathOperator{\Sym}{Sym}

\begin{document}

\title{Abelian Hall categories}
\author[Sabin Cautis]{Sabin Cautis}
\address[Sabin Cautis]{University of British Columbia \\ Vancouver BC, Canada}
\email{cautis@math.ubc.ca}

\begin{abstract}
To a quiver we associate a finite length monoidal abelian category which categorifies the corresponding preprojective K-theoretic Hall algebra of Varagnolo-Vasserot. The simples in this category provide a (dual) canonical basis of the Hall algebra. In particular, if the quiver is affine, this provides a basis for the positive half of the corresponding quantum toroidal algebra. We also show that this abelian category is naturally endowed with renormalized $\r$-matrices. 
\end{abstract}

\maketitle
\setcounter{tocdepth}{1}
\tableofcontents

\section{Introduction}

Fix a quiver $Q = (I,E)$. The main object of study in this paper is the moduli space $\cX_a$ of pairs $(\phi,x)$ where $x = (x_e)$ is a representation of $Q$ of dimension $a$ and $\phi = (\phi_i)$ is a nilpotent endomorphism of this representation.\footnote{$\cX_a$ is in fact the further quotient of this moduli space by $T = \C^\times \times \C^\times$ (see Section \ref{sec:modulispaces}) where one $\C^\times$ gives rise to a quantum grading and the other is used to define the t-structure in Theorem \ref{thm:main1}.} If we denote by $\Coh(\cX_a)$ the (derived) category of coherent sheaves on $\cX_a$ then the Hall product endows $\bigoplus_a \Coh(\cX_a)$ with a monoidal structure. The following is the main result of this paper.

\begin{Theorem}\label{thm:main1}
There exists a natural t-structure on each $\Coh(\cX_a)$ with respect to which the Hall product is t-exact. The hearts $\Coh(\cX_a)^\heartsuit$ of these t-structures yield a finite length, abelian, monoidal category which we denote $\H_Q$.   
\end{Theorem}

We call this the Koszul-perverse t-structure since it combines the perverse-coherent t-structure on nilpotent cones $\cN_a$ with a t-structure on the fibers of the natural map $\cX_a \to \cN_a$ which is related to Koszul duality. It is defined in a similar way to the t-structure from \cite{CW23} though the proof of t-exactness of the Hall product is a bit different.

\subsection{Preprojective Hall algebras and dimensional reduction}\label{sec:preprojective}

To the quiver $Q = (I,E)$ we can associate the following sets of edges
\begin{itemize}
\item $E^\ell$: $E$ together with one additional self loop at each vertex,
\item $E'$: $E$ but with edges having the opposite orientation,
\item $\oE := E \sqcup E'$,
\item $\tE := E^\ell \sqcup E' = E \sqcup E^{\prime \ell}$.
\end{itemize}
Now consider the larger stacks $\cY_a \supset \cX_a$ defined like $\cX_a$ but without the nilpotency condition on $\phi$. The stacks $\cY_a$ are equivalently\footnote{We thank Olivier Schiffmann for pointing this out.} moduli spaces of representations of $Q^\ell = (I,E^\ell)$ subject to the natural commutation relations associated with each edge $e \in E$. On the other hand, what is usually studied in the context of K-theoretic Hall algebras is the Hall algebra associated to the preprojective algebra $\Pi_Q$. In other words, one uses the moduli spaces $\Rep_a(\Pi_Q)$ of representations of $\oQ = (I,\oE)$ subject to the preprojective commutation relation. We will now explain why the two corresponding Hall categories are equivalent. 

The stack $\cY_a$ can be realized as the (derived) vanishing locus of a section $s$ of a bundle over the moduli stack $\Rep_a(Q^\ell)$ of representations of $Q^\ell$ without relations. By dimensional reduction \cite{Isik13,Hir17} our category $\Coh(\cY_a)$ is equivalent to the category of graded matrix factorizations $\MF^\gr(\Rep_a(\tQ),W)$ where $\tQ = (I,\tE)$ and $W$ is the potential associated with the section $s$. 

Similarly, the stack $\Rep_a(\Pi_Q)$ can be realized as the (derived) vanishing locus of a section of a bundle over $\Rep_a(\oQ)$. Dimensional reduction once again provides an equivalence with $\MF(\Rep_a(\tQ),W)$. Combining these equivalences gives an equivalence between $\Coh(\cY_a)$ and $\Coh(\Rep_a(\Pi_Q))$. Moreover, by \cite[Prop. 2.3.9]{Toda24} (see also \cite[Appendix ]{AG15}), objects in $\Coh(\cY_a)$ supported on $\cX_a \subset \cY_a$ correspond to objects in $\Coh(\Rep_a(\Pi_Q))$ with nilpotent singular support (in the sense of Arinkin-Gaitsgory). 

In this paper we prefer to work with $\Coh(\cX_a)$ rather than $\Coh_\Nilp(\Rep_a(\Pi_Q))$ because it is easier to impose the nilpotency condition on $\cY_a$ than the singular support condition on $\Rep_a(\Pi_Q)$. As a result it is also easier to define the t-structure from Theorem \ref{thm:main1} on $\cX_a$. Alternatively, we could have worked with $\MF^\gr(\Rep_a(\tQ),W)$ but we prefer the language of coherent sheaves to that of matrix factorizations. 

\subsection{Renormalized $\r$-matrices}

The theory of renormalized $\r$-matrices was introduced in \cite{KKK18} in the context of quiver Hecke algebras. It was subsequently used in \cite{KKKO18} to study monoidal categorifications of cluster algebras. We review some of the basic definitions and properties in Section \ref{sec:r-matrices}. The following is the other main result in this paper. 

\begin{Theorem}\label{thm:main2}
The abelian categories $\H_Q$ carry renormalized $\r$-matrices. 
\end{Theorem}

In \cite{CW18} we explained how the chiral structure of the affine Grassmannian can be used to define renormalized $\r$-matrices for the coherent Satake category. The same argument was used in \cite{CW23} to define renormalized $\r$-matrices in categories which categorify Coulomb branches. In this paper the categories $\H_Q$ do not carry a chiral structure. Nevertheless, we show in Section \ref{sec:diagonal} that they satisfy a similar property (cf. diagram \eqref{eq:tau}) which suffices for the purposes of defining renormalized $\r$-matrices (which is done in Section \ref{sec:r}). 

The existence of renormalized $\r$-matrices implies (under some technical conditions) various structural results (cf. \cite[Thm. 3.2]{KKKO15}). Such results, together with the connection between $\H_Q$ and $U_q(L\g)^+$ discussed in Section \ref{sec:quantumgroups}, give conceptual explanations for identities such as the Serre relation $[f_i,[f_i,f_{i+1}]_q]_{q^{-1}} = 0$ in $U_q(L\sl_n)^+$. 

Inspired by the quantum affine Schur-Weyl duality functors defined in \cite{KKK18}, renormalized $\r$-matrices also facilitate the construction of functors from module categories over quiver Hecke algebras (KLR algebras) to $\H_Q$. Such a functor was defined (by different methods) in \cite{SVV22} when $Q$ is the $A_1$ quiver and the quiver Hecke algebra is of affine type $A_1^{(1)}$.

\subsection{Quantum groups}\label{sec:quantumgroups}

Suppose $Q$ is a quiver of finite or affine type. Denote by $U_q(L\g)$ the integral form of the quantum enveloping algebra of the loop algebra of the corresponding Kac-Moody algebra. In other words, a quantum affine algebra if $Q$ is of finite type and a quantum toroidal algebra (double affine algebra) if $Q$ is of affine type. We denote by $U_q(L\g)^+$ the positive half of this algebra. 

It follows from \cite[Thm. A]{VV22} that these algebras are isomorphic to the corresponding K-theoretic Hall algebras (at least when $Q$ is not of type $A_1^{(1)}$). The following is then a consequence of Theorem \ref{thm:main1}. 

\begin{Corollary}\label{cor:main1}
The classes of simples in $\H_Q$ provide a positive integral basis for $U_q(L\g)^+$. 
\end{Corollary}

If $Q$ is of finite type then the resulting integral basis is the so called dual canonical basis of (the positive half of) quantum affine algebras. But, as far as we know, if $Q$ is of affine type then the resulting basis for (the positive half of) quantum toroidal algebras is new. 

The ingredients usually used to define the (dual) canonical basis is a PBW basis and the bar involution. Roughly, the (dual) canonical basis is determined by being invariant under the bar involution and by being upper triangular with respect to the PBW basis. In the context of quantum affine algebras such PBW basis go back to \cite{Beck94,BCP99}. In Section \ref{sec:examples4} we briefly illustrate how the generators of this basis correspond to natural simple objects in $\H_Q$ when $Q$ is a quiver of type $A$. 

More generally, based on discussions with P\u{a}durariu, we expect that $\H_Q$ has a semi-orthogonal decomposition (cf. \cite[Thm. 1.1]{Pad24}). At the level of K-theory this semi-orthogonal decomposition provides the PBW basis. In fact, we expect that $\H_Q$ is a properly stratified category. This is known to be the case when $Q$ is a single vertex (with no edges) based on work of Bezrukavnikov \cite{Bez03}. 

As for the bar involution, in the future paper \cite{Cau25} we will define an autoequivalence of $\H_Q$ which defines the bar involution at a categorical level (cf. discussion in Section \ref{sec:symmetries}).

\subsection{Commuting stacks}\label{sec:commutingstacks}

The case when $Q$ is of type $A_1^{(1)}$ (the Jordan quiver) is of particular interest. In this case $\cX_a$ is a (semi-nilpotent) commuting stack and (the Drinfeld double of) the K-theoretic Hall algebra has been studied extensively \cite{SV13,Neg22,GN25}. Depending on the context it is sometimes called the quantum toroidal $\gl_1$ algebra, elliptic Hall algebra, the spherical double affine algebra of $\gl_\infty$, etc. 

Among these, in a sequence of papers starting with \cite{GNR21} the authors conjecture various precise relationships between affine Hecke categories, Khovanov-Rozansky link homology and $\Coh(\cX_a)$. In particular, it was suggested (cf. \cite[Problem 1.1]{GN25}) that there should exist a functor
$$\Tr({\rm ASBim}_a)^{\Kar} \to \Coh(\cX_a)$$
where ${\rm ASBim}$ is the affine Soergel bimodule category, $\Tr$ denotes its trace and $\Kar$ its idempotent completion. It was shown in \cite{GN23} that the left side is generated by (the products) of certain objects $H_{m,n}$ where $(m,n) \in \N \times \Z$. It is subsequently conjectured in \cite{GN25} that the images $\cH_{m,n}$ of these objects are given by certain explicit formulas (\cite[Eq. 20]{GN25}). Furthermore, the pullback of this class to the Hilbert scheme of points $\Hilb^m(\C^2,\C)$ recovers the Khovanov-Rozansky homology of the $(m,n)$ torus link (cf. \cite{GN15}).  

On the other hand, for every $(m,n)$ as above there exists a natural simple object $\cP_{m,n} \in \Coh(\cX_m)^\heartsuit$ (see Section \ref{subsec:simples}). These objects should categorify the generators $P_{m,n}$ of the elliptic Hall algebra of Burban and Schiffmann \cite{BS12}. It is natural to expect that $\cH_{m,n} \cong \cP_{m,n}$. In the forthcoming paper \cite{CPT25} we will use the semi-orthogonal decomposition from \cite{Pad23} and the t-structure from Theorem \ref{thm:main1} to study such questions and, more generally, to understand the derived category of commuting stacks in greater depth. 

The appearance of Hecke categories in this context is also related to the geometric Langlands program and the study of affine character sheaves \cite{BNP17,LNY25}. Building on the case of the Jordan quiver one can take $Q$ to be the quiver with one vertex and $g$ self loops. In this case, the K-theoretic Hall algebra can be identified with the (universal) spherical Hall algebra of a smooth projective curve of genus $g$ \cite{SV12}. One can try to understand this in the Langlands context as a relation between $\H_Q$ and the category of nilpotent sheaves on the moduli stack of bundles on a genus $g$ curve. 

\subsection{Further symmetries}\label{sec:symmetries}

In \cite{Cau25} we will study two monoidal equivalences 
\begin{equation}\label{eq:equivalences}
\bkappa_\Omega: \H_Q \rightleftarrows  \H_{Q'}: \varpi^*
\end{equation}
where $Q' = (I,E')$ is the quiver $Q$ with the opposite orientation. The equivalence $\bkappa_\Omega$ is a contra-variant functor induced by linear Koszul duality \cite{MR10}. It depends on a choice of dualizing object $\Omega$ which we fix. In the case of nilpotent cones ({\it i.e.} when $Q$ has no edges) $\bkappa_\Omega$ reduces to the duality functor induced by $\Omega$.
 
The equivalence $\varpi^*$ is pullback with respect to an isomorphism of stacks $\varpi: \cX_Q \to \cX_{Q^*}$. It is a categorical analogue of the isomorphism $\varpi^*$ of \cite[Sec. 9.8]{Lus98} which acts on the K-theory of the Steinberg variety (and thus induces an involution of the affine Hecke algebra). In the case of nilpotent cones ({\it i.e.} when $Q$ has no edges) $\varpi$ reduces to the involution denoted $\tau$ from \cite[Sect. 4.2]{FF21}. 

The two equivalences in \eqref{eq:equivalences} can be used to define the bar auto-equivalence $\overline{\cF} := \varpi^* \circ \bkappa_\Omega(\cF)$ on $\H_Q$. When $Q$ has no edges, this auto-equivalence is denoted $\iota$ in \cite[Sect. 4.2]{FF21}. In \cite{Cau25} we will check that, up to loop shifts, the bar equivalence fixes simples in $\H_Q$.

\subsection{Acknowledgements}
We thank Dylan Butson, Tudor P\u{a}durariu, Olivier Schiffmann, Eric Vasserot and Harold Williams for many insightful discussions on Hall algebras and related topics and again P\u{a}durariu for helpful feedback on an earlier version of the paper. This research was supported by NSERC Discovery Grant 2019-03961, the Fondation des Sciences Math\'ematiques de Paris and Universit\'e d'Orsay. 

\section{Hall categories}

In this paper we will work with spaces $X$ that are derived geometric stacks of finite type. In fact, all stacks will be quotient stacks so it is possible to work with normal (derived) schemes but the language of stacks clarifies and simplifies the exposition. If $X$ is such a stack we will denote by $\QCoh(X)$ the $\infty$-category of quasi-coherent sheaves on $X$ and by $\Coh(X)$ the subcategory of coherent sheaves (i.e. bounded with coherent cohomology). 

\subsection{Derived kernels}\label{sec:vanloci}

Consider a morphism of vector bundles $s: V \to W$ over some base $X$. The {\it derived kernel} $V^s$ of $s$ is given by the following fiber product 
\begin{equation}\label{eq:van-square}
\xymatrix{
V^s \ar[r]^{i_s} \ar[d]^{i_0} & V \ar[d]^{\Gamma_0} \\
V \ar[r]^-{\Gamma_s} & V \times_X W}
\end{equation}
where $\Gamma_s$ is the graph of $s$ and $\Gamma_0$ the graph of the zero map $0: V \to W$. The classical locus $(V^s)^\cl \subset V^s$ consists of the classical kernel of $V$. If $s$ is surjective then $V^s = (V^s)^\cl$ but in general $V^s$ will contain some non-trivial derived structure. Derived kernels play an important role because the moduli spaces $\cX_a$ defined below are of this form. We collect some technical results about derived kernels in the appendix. 

\subsection{Moduli spaces}\label{sec:modulispaces}

Given a finite quiver $Q = (I,E)$ denote by $Q^{\ell} = (I,E^{\ell})$ the quiver with an additional self loop at each vertex (cf. Section \ref{sec:preprojective}). For $a = (a_i) \in \N^I$ fix a vector space $V_a := \bigoplus_{i \in I} V_i$ where $\dim(V_i) = a_i$. Then the moduli space of representations of $Q^\ell$ is 
$$\bM_a := \{(\phi,x): \phi = (\phi_i) \in \bigoplus_{i \in I} \End(V_i), x = (x_e) \in \bigoplus_{e \in E} \Hom(V_{t(e)},V_{h(e)})\}$$
where, for an edge $e$, we denote by $h(e)$ and $t(e)$ its head and tail vertices. Denote by $M_a \subset \bM_a$ the subspace where every $\phi_i$ is nilpotent. The case when $Q$ has no edges plays a special role. In this case $M_a$ is just the product of nilpotent cones and we denote it $N_a$. In general we have natural projection maps $p_a: M_a \to N_a$ so that $M_a$ has the structure of a vector bundle over $N_a$. 

The spaces above carry a natural action of $G_a = \prod_{i \in I} GL_{a_i}$. We extend this to an action of $\hG_a := G_a \times T$ where $T = \C^\times \times \C^\times$ acts by $(t_1,t_2) \cdot (\phi,x) = (t_1^2 \phi, t_1^{-1} t_2 x)$. These two factors of $T$ are called the loop and the scaling $\C^\times$. We denote by $\{1\}$ and $\la 1 \ra$ the corresponding twists (i.e. tensoring with the corresponding characters) and refer to these as the loop and scaling shifts. 

Next we consider the locus in $M_a$ where $\phi$ commutes with $x$. This is given by the (derived) kernel $M_a^{s_a}$ where $s_a: M_a \to M_a \{2\}$ is the commutator of $\phi$ and $x$. In other words, we have the following  fiber product 
\begin{equation}\label{eq:M}
\xymatrix{
M_a^{s_a} \ar[rr]^{i_a} \ar[d]^{i_0} & & M_a \ar[d]^{\Gamma_0} \\
M_a \ar[rr]^-{\Gamma_{s_a}} & & M_a \times_{N_a} M_a \{2\}}
\end{equation}
where we abbreviate $i_{s_a}$ as $i_a$. We abuse notation by writing $M_a \{2\}$ for the scheme with a twisted loop $\C^\times$-action.

Note that we still have an action of $\hG_a$ on $M_a^{s_a}$ since all the maps above are $\hG_a$-equivariant. To simplify notation we denote the quotient stacks $[N_a/\hG_a]$, $[M_a/\hG_a]$ and $[M^{s_a}_a/\hG_a]$ by $\cN_a$, $\cM_a$ and $\cX_a = \cM_a^{s_a}$ respectively. Thus $\Coh(\cX_a) = \Coh^{\hG_a}(M_a^{s_a})$. We also write $\cX_{Q} = \sqcup_a \cX_a$. 

\subsection{Convolution} 

We now explain how to define the Hall product on $\Coh(\cX_Q)$. For $a,b \in \N^I$ denote $V_a = \oplus_{i \in I} V_i$ and $V'_b = \oplus_{i \in I} V'_i$ with $\dim(V_i)=a_i$ and $\dim(V'_i)=b_i$. Denote by $P_{a,b} \subset G_{a+b}$ the parabolic which fixes $V_a \subset V_a \oplus V'_b$ and by $\mfp_{a,b} \subset \bigoplus_{i \in I} \End(V_i \oplus V'_i)$ its Lie algebra. Consider 
\begin{equation}\label{eq:M3}
\bM_{a,b} := \{(\phi,x): (\phi_i) \in \mfp_{a,b}, (x_e) \in \bigoplus_{e \in E} \Hom(V_{t(e)} \oplus V'_{t(e)}, V_{h(e)} \oplus V'_{h(e)}) \text{ fixing } V_a \}
\end{equation}
and $M_{a,b} \subset \bM_{a,b}$ defined by requiring that every $\phi_i$ is nilpotent. As before, if $Q$ has no edges we denote these by $N_{a,b}$ and $\bN_{a,b}$. Note that $p_{a,b}: M_{a,b} \to N_{a,b}$ is naturally a vector bundle. 

The locus where $\phi$ and $x$ commute is the fiber product 
\begin{equation}\label{eq:M2}
\xymatrix{
M_{a,b}^{s_{a,b}} \ar[rr]^{i_{a,b}} \ar[d]^{i_0} & & M_{a,b} \ar[d]^{\Gamma_0} \\
M_{a,b} \ar[rr]^-{\Gamma_{s_{a,b}}} & & M_{a,b} \times_{N_{a,b}} M_{a,b} \{2\} }
\end{equation}
where $s_{a,b}: M_{a,b} \to M_{a,b} \{2\}$ is given by the commutator of $\phi$ and $x$. The group $\hP_{a,b} := P_{a,b} \times T$ acts on the diagram above and we denote $\cM_{a,b} := [M_{a,b}/\hP_{a,b}]$ and $\cM_{a,b}^{s_{a,b}} := [M_{a,b}^{s_{a,b}}/\hP_{a,b}]$.

We have natural maps $M_a \times M_b \xleftarrow{\tpi_{a,b}} M_{a,b} \xrightarrow{\tm_{a,b}} M_{a+b}$ which are equivariant with respect to the morphisms $\hG_a \times \hG_b \leftarrow \hP_{a,b} \rightarrow \hG_{a+b}$. One can check that these maps commute with $(s_a,s_b), s_{a,b}$ and $s_{a+b}$ respectively. Thus, by Corollary \ref{cor:van2}, we get the following induced maps
$$(M_a \times M_b)^{(s_a,s_b)} = M_a^{s_a} \times M_b^{s_b} \xleftarrow{\pi_{a,b}} M_{a,b}^{s_{a,b}} \xrightarrow{m_{a,b}} M_{a+b}^{s_{a+b}}.$$
Taking quotients we also obtain maps 
$$\cM_a^{s_a} \times \cM_b^{s_b} \xleftarrow{\pi_{a,b}} \cM_{a,b}^{s_{a,b}} \xrightarrow{m_{a,b}} \cM_{a+b}^{s_{a+b}}.$$

Now $\tpi_{a,b}: \cN_{a,b} \to \cN_a \times \cN_b$ has coherent pullback and, as maps of bundles over $\cN_{a,b}$, the map $\cM_{a,b} \to \tpi_{a,b}^*(\cM_a \times \cM_b)$ is surjective. Thus Corollary \ref{cor:van2} implies that $\pi_{a,b}^*$ preserves coherence. On the other hand, $\tm_{a,b}: \cN_{a,b} \to \cN_{a+b}$ is proper and the map $\cM_{a,b} \to \tm_{a,b}^* \cM_{a+b}$ is injective so, again by Corollary \ref{cor:van2}, $m_{a,b*}$ preserves coherence. 

\begin{Definition}\label{def:Hall}
The Hall product of $\cF_a \in \Coh(\cX_a)$ and $\cF_b \in \Coh(\cX_b)$ is defined as
\begin{equation}\label{eq:defconv}
\cF_a * \cF_b := m_{a,b*} \pi_{a,b}^* (\cF_a \boxtimes \cF_b) [u_{a,b}]\{-u_{a,b}-v_{a,b}\}
\end{equation}
where $u_{a,b} := \sum_{i \in I} a_ib_i$ and $v_{a,b} := \sum_{e \in E} a_{t(e)}b_{h(e)}$. 
\end{Definition}

The extra shifts in \eqref{eq:defconv} are largely to make subsequent results and statements cleaner. For example, the calculation in Corollary \ref{cor:calc3} would be asymmetric without the loop shift in \eqref{eq:defconv}. 

\begin{Proposition}\label{prop:associative}
The product defined by \eqref{eq:defconv} is associative. 
\end{Proposition}
\begin{proof}
For $a,b,c \in \N^I$ consider the following diagram where the top left square is Cartesian
\begin{equation}\label{eq:assoc1}
\xymatrix{
\cM_{a,b,c} \ar[r]^{\tm} \ar[d]^{\tpi} & \cM_{a+b,c} \ar[r]^{\tm_{a+b,c}} \ar[d]^{\tpi_{a+b,c}} & \cM_{a+b+c} \\
\cM_{a,b} \times \cM_c \ar[r]^{\tm_{a,b} \times \id} \ar[d]^{\tpi_{a,b} \times \id} & \cM_{a+b} \times \cM_c \\
\cM_a \times \cM_b \times \cM_c && }
\end{equation}
Writing as before $V_a = \oplus_{i \in I} V_i$, $V'_b = \oplus_{i \in I} V'_i$ and $V_c = \oplus_{i \in I} V''_i$ with $\dim(V_i) = a_i, \dim(V'_i) = b_i, \dim(V''_i) = c_i$ we can identify $\cM_{a,b,c}$ with the quotient stack $[M_{a,b,c}/\hP_{a,b,c}]$ where
\begin{align*}
M_{a,b,c} = \{(\phi,x):  & \phi \in \mfp_{a,b,c} \text{ nilpotent}, x_e \in \Hom(V_{t(e)} \oplus V'_{t(e)} \oplus V''_{t(e)}, V_{h(e)} \oplus V'_{h(e)} \oplus V''_{h(e)}) \\
& \text{ fixing $V_a$ and $V_a \oplus V'_b$ }\}
\end{align*}
Here $\hP_{a,b,c} = P_{a,b,c} \times T$ where $P_{a,b,c} \subset G_{a+b+c}$ is the parabolic which fixes the flag $V_a \subset V_a \oplus V'_b \subset V_a \oplus V'_b \oplus V''_c$ and $\mfp_{a,b,c}$ is its Lie algebra.

Moreover, the natural commutator maps $s_a, s_b, s_{a,b}$ etc. all commute with the maps in \eqref{eq:assoc1} which gives us the commutative diagram
\begin{equation}\label{eq:assoc2}
\xymatrix{
\cM_{a,b,c}^{s_{a,b,c}} \ar[r]^{m} \ar[d]^{\pi} & \cM^{s_{a+b},c}_{a+b,c} \ar[r]^{m_{a+b,c}} \ar[d]^{\pi_{a+b,c}} & \cM^{s_{a+b+c}}_{a+b+c} \\
\cM_{a,b}^{s_{a,b}} \times \cM^{s_c}_c \ar[r]^{m_{a,b} \times \id} \ar[d]^{\pi_{a,b} \times \id} & \cM^{s_{a+b}}_{a+b} \times \cM^{s_c}_c \\
\cM^{s_a}_a \times \cM^{s_b}_b \times \cM^{s_c}_c && }
\end{equation}
By Proposition \ref{prop:van} the top left square of \eqref{eq:assoc2} is still Cartesian. Thus, for $\cF_a \in \Coh(\cX_a), \cF_b \in \Coh(\cX_b), \cF_c \in \Coh(\cX_c)$ we get the following sequence of natural isomorphisms
\begin{align*}
(\cF_a * \cF_b) * \cF_c 
&\cong m_{a+b,c*} \pi_{a+b,c}^* (m_{a,b} \times \id)_* (\pi_{a,b} \times \id)^* (\cF_a \boxtimes \cF_b \boxtimes \cF_c) \\
&\cong m_{a+b,c*} \pi^* m_* (\pi_{a,b} \times \id)^* (\cF_a \boxtimes \cF_b \boxtimes \cF_c) \\
&\cong m_{a,b,c*} \pi^*_{a,b,c} (\cF_a \boxtimes \cF_b \boxtimes \cF_c)
\end{align*}
where the second isomorphism is by base change and $\pi_{a,b,c}, m_{a,b,c}$ denote the compositions along the left and upper sides of diagram \eqref{eq:assoc2}.

A similar arguments shows that we likewise have a natural isomorphism
$$\cF_a * (\cF_b * \cF_c) \cong m_{a,b,c*} \pi^*_{a,b,c} (\cF_a \boxtimes \cF_b \boxtimes \cF_c).$$
\end{proof}

\subsection{Other flavors and remarks}\label{sec:flavors}

Recall the torus $T \subset \hG$ which consists of the scaling and loop $\C^\times$s. The scaling $\C^\times$ plays a crucial role in defining the Koszul-perverse t-structure and it is important that it act with weight one on the fibers of  $M_a \to N_a$. The loop $\C^\times$ plays a role in the theory of renormalized $\r$-matrices as discussed in Section \ref{sec:r-matrices}. At the level of K-theory it corresponds to the quantum parameter $q$ that appears, for example, in quantum groups. The fact that it acts with weight $-1$ on $M_a$ is largely for aesthetic reasons so that certain results and relations are more symmetric. If one omits the loop $\C^\times$ then one looses this extra grading but the rest of the results remain unaffected. 

If one wants to define the Koszul-perverse t-structure without the scaling $\C^\times$ one needs to find another central cocharacter $\eta: \C^\times \to \hG$ (as in Section \ref{sec:koszul}) which acts like the scaling $\C^\times$ on $M_a$. This is often but not always possible. For example, it is possible if $Q$ is a finite type quiver but not when $Q$ is the Jordan quiver. If such an alternative $\eta$ exists than the scaling $\C^\times$ is redundant in the sense that $\Coh(\cX_a) \cong \bigoplus_{n \in \Z} \Coh(\cX'_a)$ where $\cX'_a$ is the corresponding stack without the scaling $\C^\times$.

On the other hand, one can also enlarge $T \subset \hG$. For example, for an edge $e \in E$ in the quiver, one can add a copy of $\C^\times$ which acts trivially on every $\phi_i$ and $x_{e'}$ when $e' \ne e$ but acts with weight one on $x_e$. Doing this allows one to generalize the symmetries discussed in Section \ref{sec:symmetries}. More precisely, if we denote by $Q^e$ the same quiver but with the edge labeled by $e$ having the opposite orientation, then this extra $\C^\times$ allows one to define equivalences 
$$\H_{Q} \rightleftarrows \H_{Q^e}$$
analogous to those in \eqref{eq:equivalences}.

Lastly, it is not hard to see that the action of the diagonal $\C^\times \subset G_a \subset \hG_a$ on $\cX_a$ is trivial. This provides a weight decomposition $\Coh(\cX_a) = \oplus_{n \in \Z} \Coh(\cX_a)_n$. In particular, $\Coh(\cX_Q) = \oplus_a \Coh(\cX_a)$ is naturally graded by $(a,n) \in \N^I \times \Z$. 

\section{Koszul-perverse t-structure}

In this section we define a t-structure on $\Coh(\cX_a)$ and discuss some of its properties. 

\subsection{Koszul t-structures}\label{sec:koszul}

Fix an affine group scheme $\hG$ acting on a classical scheme $X$ and suppose $\Coh^\hG(X)$ carries a given t-structure. Suppose $\eta: \C^\times \to \hG$ is a central cocharacter which acts trivially on $X$. This gives us a weight decomposition $\Coh^\hG(X) = \bigoplus_{n \in \Z} \Coh^\hG_n(X)$. 

Given a $\hG$-equivariant morphism $\pi: Y \to X$, we say that a t-structure on $\Coh^\hG(Y)$ is Koszul over $X$ if $- \otimes \pi^*(\cV) [-n]$ is t-exact for any $n \in \Z$ and locally free sheaf $\cV \in \Coh^{\hG}_n(X)$ (cf. \cite[Def. 3.8]{CW23}). A t-structure on $\Coh^\hG(X)$ is Koszul over $X$ if it is Koszul for the identity map $X \to X$.

\begin{Proposition}\label{prop:koszul}
Suppose $V, W \in \Coh^{\hG}_1(X)$ are $\hG$-equivariant vector bundles and $s: V \to W$. Fix a Koszul t-structure on $\Coh^\hG(X)$. Then 
\begin{enumerate}
\item there exist t-structures on $\Coh^{\hG}(V)$ and $\Coh^{\hG}(W)$ which are Koszul over $X$ and are determined by the t-exactness of $\sigma_V^*$ and $\sigma_W^*$ respectively,
\item there exists a t-structure on $\Coh^{\hG}(V^s)$ which is Koszul over $X$ and is determined by the t-exactness of $i_{s*}$.
\end{enumerate}
Here $\sigma_V: X \to V$ and $\sigma_W: X \to W$ are the zero sections and $i_s: V^s \to V$ is the map from \eqref{eq:van-square}. Both these t-structures are bounded and finite length if the t-structure on $\Coh^{\hG}(X)$ is. 
\end{Proposition}
\begin{proof}
The existence of the first two t-structures (and their properties) follows from \cite[Prop. 3.16]{CW23} while the last follows from \cite[Thm. 3.1]{CW23}. 
\end{proof}

\subsection{Koszul-perverse coherent sheaves}

We start with the perverse (coherent) t-structure on $\Coh^{\hG_a}(N_a) = \Coh(\cN_a)$ and take $\eta$ to be the scaling $\C^\times \subset T \subset \hG$. This t-structure is local in the sense that tensoring with any locally free sheaf is t-exact. Any such local t-structure can be twisted so that it becomes Koszul. This is done by shifting the local t-structure on $\Coh_n^{\hG_a}(N_a)$ by $[n]$ so that $[1]\la -1 \ra$ becomes t-exact (we call $[1]\la -1 \ra$ the Koszul shift). Applying this to the perverse t-structure we obtain the Koszul-perverse t-structure on $\Coh(\cN_a)$ and its heart $\Coh(\cN_a)^{\heartsuit}$.

\begin{Definition}\label{def:KP}
Consider the natural maps 
$$\cX_a = [M_a^{s_a}/\hG_a] \xrightarrow{i_a} [M_a/\hG_a] \xleftarrow{\sigma_a} [N_a/\hG_a].$$
By Proposition \ref{prop:koszul} there exists a unique t-structure on $\Coh(\cX_a)$ such that
$$\sigma_a^* i_{a*}: \Coh(\cX_a) \to \Coh(\cN_a)$$
is t-exact with respect to the Koszul-perverse t-structure on $\Coh(\cN_a)$. We call this the {Koszul-perverse t-structure} on $\Coh(\cX_a)$ and denote its heart $\Coh(\cX_a)^\heartsuit$. 
\end{Definition}

Note that by Proposition \ref{prop:koszul} there is also a Koszul-perverse t-structure on $\Coh([M_a/\hG_a])$ determined by the t-exactness of 
$$\sigma_a^*: \Coh([M_a/\hG_a]) \to \Coh(\cN_a).$$
We denote the heart of this t-structure $\Coh([M_a/\hG_a])^\heartsuit$. It follows from \cite[Thm. 3.31]{CW23} that $\Coh([M_a/\hG_a])^\heartsuit$ and $\Coh(\cX_a)^\heartsuit$ are bounded and of finite length.

\subsection{t-exactness of convolution}

\begin{Proposition}\label{prop:t-exact}
The Hall product on $\Coh(\cX_Q)$ restricts to a product on $\Coh(\cX_Q)^\heartsuit$. 
\end{Proposition}
\begin{proof}
It is known that convolution preserves perverse coherent sheaves on $\cN_a$ in the same way that convolution on the affine Grassmannian preserves perverse coherent sheaves (e.g. \cite{BFM05}). Since the $\Coh(\cN_a)^\heartsuit$ is obtained by shifting the category of perverse coherent sheaves it is easy to see that convolution preserves $\Coh(\cN_a)^\heartsuit$. We will now reduce the general statement to this case.  

To simplify exposition we omit the $\{-\}$ shifts since they are t-exact for the Koszul-perverse t-structure and thus play no role. Suppose $\cF_a \in \Coh(\cM^{s_a}_a)^\heartsuit$ and $\cF_b \in \Coh(\cM^{s_b}_b)^\heartsuit$. To show $\cF_a * \cF_b \in \Coh(\cM_{a+b}^{s_{a+b}})^\heartsuit$ it suffices to show that
$$\sigma_{a+b}^* i_{a+b*}(\cF_a * \cF_b) \cong \sigma_{a+b}^* i_{a+b*} m_{a,b*} \pi_{a,b}^* (\cF_a \boxtimes \cF_b) [u_{a,b}] \cong \sigma_{a+b}^* \tm_{a,b*} i_{a,b*} \pi_{a,b}^* (\cF_a \boxtimes \cF_b) [u_{a,b}]$$
belongs to $\Coh(\cN_{a+b})^\heartsuit$.

Now consider the following commutative diagram
\begin{equation}\label{eq:xy1}
\xymatrix{
\cM_a^{s_a} \times \cM_b^{s_b} \ar[d]^{i_a \times i_b} & & \cM_{a,b}^{s_{a,b}} \ar[ll]_-{\pi_{a,b}} \ar[r]^{m_{a,b}} \ar[d]^{i_{a,b}} & \cM_{a+b}^{s_{a+b}} \ar[d]^{i_{a+b}} \\
\cM_a \times \cM_b & &\cM_{a,b} \ar[ll]_-{\tpi_{a,b}} \ar[r]^{\tm_{a,b}} &  \cM_{a+b}
}
\end{equation}
The left square in (\ref{eq:xy1}) is not Cartesian and we cannot apply base change directly. Instead consider the following commutative diagram. 
\begin{equation}\label{eq:local4}
\xymatrix{
\cM_{a,b}^{s_{a,b}} \ar[r] \ar[dd]^{i_{a,b}} & \cM_{a,b}^{f \circ s_{a,b}} \ar[r]^-{\bar{f}} \ar[d]^{i_{f \circ s_{a,b}}} & \pi_{a,b}^*(\cM_a^{s_a} \times \cM_b^{s_b}) \ar[r] \ar[d]^{\pi_{a,b}^*(i_a \times i_b)} & \cM_a^{s_a} \times \cM_b^{s_b} \ar[d]^{i_a \times i_b} \\
 & \cM_{a,b} \ar[d]^{\Gamma_0} \ar[r]^-{f} & \pi_{a,b}^*(\cM_a \times \cM_b) \ar[r] \ar[d] & \cM_a \times \cM_b \ar[d] \\ 
\cM_{a,b} \ar[r]^{\Delta} & \cM_{a,b} \times_{\pi_{a,b}^*(\cM_a \times \cM_b)} \cM_{a,b} & \cN_{a,b} \ar[r]^{\pi_{a,b}} & \cN_a \times \cN_b }
\end{equation}
where $f: \cM_{a,b} \to \pi_{a,b}^*(\cM_a \times \cM_b)$ is the map of bundles over $\cN_{a,b}$ induced by $\tpi_{a,b}$. The right-most squares are Cartesian by \eqref{eq:van-square4}, the middle square is Cartesian by \eqref{eq:van-square3B} and the left rectangle is Cartesian by \eqref{eq:van-square3A}. 

The composition along the top row of \eqref{eq:local4} is $\pi_{a,b}$ and along the middle row $\tpi_{a,b}$. Thus, by base change, we get that 
$$i_{a,b*} \pi_{a,b}^*(\cF_a \boxtimes \cF_b) \cong \Delta^* \Gamma_{0*} \tpi_{a,b}^* (i_a \times i_b)_*(\cF_a \boxtimes \cF_b).$$
Now $(i_a \times i_b)_* (\cF_a \boxtimes \cF_b) \in \Coh(\cM_a \times \cM_b)^\heartsuit$. Moreover, any simple in $\Coh(\cM_a \times \cM_b)^\heartsuit$ is the pullback $(p_a \times p_b)^*(\cG)$ of a simple $\cG \in \Coh(\cN_a \times \cN_b)^\heartsuit$ via the natural projection map $(p_a \times p_b): \cM_a \times \cM_b \to \cN_a \times \cN_b$. Thus it suffices to show that 
$$\sigma_{a+b}^* \tm_{a,b*} \Delta^* \Gamma_{0*} \tpi_{a,b}^* (p_a \times p_b)^* (\cG) [u_{a,b}] \in \Coh(\cN_{a+b})^\heartsuit$$
for any $\cG \in \Coh(\cN_a \times \cN_b)^\heartsuit$. 
 
Now, for any sheaf $\cF$, $\Delta^* \Gamma_{0*}(\cF)$ has a filtration with quotients of the form $\cF \otimes \Lambda^k(\cK^\vee_{a,b})[k]\la -k \ra$ where $\cK_{a,b}$ is the kernel of the bundle map $f$. The bundle $\cK_{a,b}$ is the pullback of a bundle from $\cN_a \times \cN_b$. More precisely, $\cK_{a,b} \cong \tpi_{a,b}^* (p_a \times p_b)^* \cK$ where, in the notation of \eqref{eq:M3}, $\cK$ is the the bundle $\oplus_{e \in E} \Hom(V_{t(e)}', V_{h(e)})$. Thus, 
$$\tpi_{a,b}^* (p_a \times p_b)^* (\cG) \otimes \Lambda^k(\cK^\vee_{a,b})[k] \la -k \ra \cong \tpi_{a,b}^* (p_a \times p_b)^* (\cG  \otimes \Lambda^k(\cK^\vee)[k] \la -k \ra).$$
But, since $\cK$ has scaling weight one, tensoring with $\Lambda^k(\cK^\vee)[k] \la -k \ra$ is t-exact for the Koszul-perverse t-structure. Thus it suffices to show that 
$$\sigma_{a+b}^* \tm_{a,b*} \tpi_{a,b}^* (p_a \times p_b)^* (\cG) [u_{a,b}] \in \Coh(\cN_{a+b})^\heartsuit$$
for any $\cG \in \Coh(\cN_a \times \cN_b)^\heartsuit$.

We can further rewrite this as $\sigma_{a+b}^* \tm_{a,b*}  p_{a,b}^* \pi^*_{a,b} (\cG) [u_{a,b}]$ since we have a commutative square
$$\xymatrix{
\cM_{a,b} \ar[r]^{\tpi_{a,b}} \ar[d]^{p_{a,b}} & \cM_a \times \cM_b \ar[d]^{p_a \times p_b} \\
\cN_{a,b} \ar[r]^{\pi_{a,b}} & \cN_a \times \cN_b }
$$
To simplify this composition further consider the following diagram
$$\xymatrix{
& \cN_{a,b} \ar[r]^{m_{a,b}} \ar[d]^{\sigma} & \cN_{a+b} \ar[d]^{\sigma_{a+b}} \\
\cM_{a,b} \ar[rd]_{p_{a,b}} \ar[r]^i & \cM'_{a,b} \ar[r]^{\tm'_{a,b}} \ar[d]^p & \cM_{a+b} \ar[d]^{p_{a+b}} \\
& \cN_{a,b} \ar[r]^{m_{a,b}} & \cN_{a+b}}
$$
where $\cM'_{a,b}$ is by definition the fiber product. Note that both squares are Cartesian and the composition along the middle row is $\tm_{a,b}$. We can rewrite 
$$\sigma_{a+b}^* \tm_{a,b*}  p_{a,b}^* \pi^*_{a,b} (\cG) \cong \sigma_{a+b}^* \tm'_{a,b*} i_* i^* p^* \pi^*_{a,b} (\cG).$$
But $i$ is an embedding of vector bundles whose cokernel $\cK'_{a,b}$ is the pullback $p^* \pi_{a,b}^*(\cK')$ of a bundle $\cK'$ on $\cN_a \times \cN_b$. Thus, for any $\cF$, $i_*i^*(\cF)$ has a filtration with quotients of the form $\cF \otimes p^* \pi_{a,b}^* \Lambda^k(\cK^{\prime \vee})[k] \la -k \ra$ and, since $\cK'$ has scaling weight one, we can argue as before to reduce to showing that 
$$\sigma_{a+b}^* \tm'_{a,b*} p^* \pi^*_{a,b} (\cG) \in \Coh(\cN_{a+b})^\heartsuit$$
for any $\cG \in \Coh(\cN_a \times \cN_b)^\heartsuit$. But, by base change, we have
$$\sigma_{a+b}^* \tm'_{a,b*} p^* \pi^*_{a,b} (\cG) \cong m_{a,b*} \sigma^* p^* \pi^*_{a,b} (\cG) \cong m_{a,b*} \pi_{a,b}^* (\cG)$$
which belongs to $\Coh(\cN_{a+b})^\heartsuit$ since the convolution product on $\cN$ which is t-exact. 
\end{proof}

We will denote the monoidal, abelian category $(\Coh(\cX_Q)^\heartsuit, *)$ by $\H_Q$. Note that the decomposition $\Coh(\cX_Q) = \oplus_{a,n} \Coh(\cX_a)_n$ discussed in Section \ref{sec:flavors} restricts to give a decomposition $\H_Q = \oplus_{a,n} \Coh(\cX_a)^\heartsuit_n$.

\subsection{Simple objects}\label{sec:simples}

It turns out that simples in $\Coh(\cX_a)^\heartsuit$ are naturally in bijection with simples in $\Coh(\cN_a)^\heartsuit$. In particular, this indexing of simples in $\Coh(\cX_a)^\heartsuit$ is independent of the quiver $Q$. We will now describe these simples in more detail. 

A simple in $\Coh(\cX_a)^\heartsuit$ will have support whose image in $\cN_a$ is the closure $\overline{\bO}_\alpha$ of some orbit $\bO_\alpha \subset \cN_a$. Now consider the diagram in \eqref{eq:simplesA}.
\begin{equation}\label{eq:simplesA}
\xymatrix{
\cU_\alpha^\cl \ar[r]^{\iota_\alpha} \ar[d]^{\pi_\alpha} & \cU_\alpha \ar[r]^{\tj_\alpha} \ar[d] & \cX_\alpha \ar[d] \ar[r]^{\ti_\alpha} & \cX_a \ar[d] \\
\bO_\alpha & \bO_\alpha \ar[r]^{j_\alpha} & \overline{\bO}_\alpha \ar[r]^{i_\alpha} & \cN_a }
\end{equation}
The maps $i_\alpha, j_\alpha$ denote the natural inclusions and the two squares on the right are Cartesian. The classical locus $\cU_\alpha^\cl \subset \cU_\alpha$ is a bundle over $\bO_\alpha$.

By \cite[Prop. 3.27]{CW23} we have $\cF \cong \ti_{\alpha_*}(\cF')$ for some $\cF' \in \Coh(\cX_\alpha)^\heartsuit$. Moreover, the restriction $\tj_\alpha^*(\cF')$ is isomorphic to $\iota_{\alpha*} \pi^*_{\alpha}(\cF'')$ for some simple $\cF'' \in \Coh(\bO_\alpha)^\heartsuit$. Since $\bO_\alpha$ is an orbit, we must have
$$\cF'' \cong \cV [\frac12 d_\alpha] \{-\frac12 d_\alpha\} [s]\la -s \ra$$
where $\cV \in \Coh(\bO_\alpha)$ is some irreducible vector bundle on $\bO_\alpha \subset \cN_a$, $d_\alpha = \dim(\bO_\alpha)$ and $s \in \Z$.
\begin{Remark}
The shift $[s]\la -s \ra$ is a Koszul shift which can be arbitrary since $[1]\la -1 \ra$ is t-exact for a Koszul t-structure. The shift $[\frac12 d_\alpha]$ is necessary so that $\cV [\frac12 d_\alpha]$ is perverse. The shift $\{-\frac12 d_\alpha \}$ is largely cosmetic since it can be absorbed into $\cV$. Its inclusion is related to the properties of the functor $\bkappa_{\Omega}$ which will be studied in \cite{Cau25}. Note that the Koszul shift could also be absorbed into $\cV$ but then $\cV$ would have a cohomological shift and would not be, strictly speaking, a bundle.
\end{Remark}

Conversely, starting with a simple $\cF'' \in \Coh(\cX_\alpha)^\heartsuit$ as above, $\iota_{\alpha*} \pi^*_{\alpha}(\cF'')$ has a unique extension $\cF' \in \Coh(\cX_\alpha)^\heartsuit$ (cf. \cite[Thm. 3.31]{CW23}). We denote this extension $j_{!*}(-)$ although this is just notation as we have not checked that this extension on simples is actually induced by a functor. Nevertheless, putting all this together means that any simple in $\Coh(\cX_a)$ is, up to Koszul shifts $[s]\la -s \ra$, of the form
\begin{equation}\label{eq:simples3}
\cF_\cV \cong \ti_{\alpha*} \tj_{\alpha !*} \iota_{\alpha*} \pi^*_{\alpha} (\cV) [\frac12 d_\alpha] \{- \frac12 d_\alpha\}
\end{equation}
for some irreducible bundle $\cV \in \Coh(\bO_\alpha)$. Futhermore, such irreducible bundles on $\bO_\alpha$ are in bijection with irreducible representation of the stabilizer in $\hG_a$ of a point in $\bO_\alpha$. 

\subsubsection{Example: the commutating stack}\label{subsec:simples}

Following up on the discussion from Section \ref{sec:commutingstacks}, the case when $Q$ is the Jordan quiver is particularly interesting. In this case the nilpotent cone $\cN_m$ contains a dense, open orbit which is the orbit of the principal nilpotent. The stabilizer in $G_m = GL_m$ of this point in $\cN_m$ is the subgroup of elements of the form 
$$A = \left( 
\begin{matrix}
c_1 & c_2 & \dots & c_m \\
0 & c_1 & \dots & c_{m-1} \\
\vdots & \vdots & \vdots & \vdots \\
0 & 0 & \dots & c_1 
\end{matrix} \right)$$
The characters $\chi_n$ of this group are indexed by $n \in \Z$ where $\chi_n(A) = c_1^n$. It follows that for every $(m,n) \in \N \times \Z$ we get a simple object $\cP_{m,n} \in \Coh(\cX_m)^\heartsuit_n$. 

\subsection{The abelian category $\bH_Q$}\label{sec:bH}

By omitting the nilpotency condition on $\phi$ in the definition of $\cX_a = [M_a^{s_a}/\hG_a]$ we obtain the spaces $\cY_a = [\bM^{s_a}_a/\hG_a]$. Instead of working with $\Coh(\cX_a)$ we could work with the category $\Coh_{\Nilp}(\cY_a) \subset \Coh(\cY_a)$ consisting of sheaves supported on $\cX_a \subset \cY_a$. The constructions and results above readily extend to give us a Koszul-perverse t-structure on $\Coh_\Nilp(\cY_a)$ (note however that there is no such t-structure if we omit the nilpotency support condition). 

The heart $\Coh(\cY_a)^\heartsuit$ of this t-structure defines an alternative abelian Hall category $\bH_Q := \Coh_\Nilp(\cY_Q)^{\heartsuit}$ where $\cY_Q = \sqcup_a \cY_a$. The inclusion $\cX_a \subset \cY_a$ induces a natural exact, faithful functor $\H_Q \to \bH_Q$. This functor identifies the simples in $\H_Q$ with those in $\bH_Q$. In particular, $\H_Q$ and $\bH_Q$ categorify the same algebra. While in this paper we only work with $\H_Q$ in order to slightly simplify the exposition, in some other contexts, such as when studying the relationship to KLR-module categories, it is more natural to work with $\bH_Q$. 

\section{Renormalized $\r$-matrices}\label{sec:r-matrices}

In this section we construct a system of renormalized $\r$-matrices for the monoidal category $\H_Q$. 

\subsection{Definitions}\label{sec:defs}

Following \cite{KKK18} (cf. \cite[Section 4]{CW18}) a system of (nonzero) renormalized $\r$-matrices in the context of a general (graded) monoidal category $\cC$ consists of nonzero maps 
$$\r_{M,N}: M*N \to N*M \{\L(M,N)\}$$
associated to any pair of objects $M,N \in \cC$. Here $\L(M,N) \in \Z$ and $\{-\}$ denotes the grading shift in $\cC$. These maps are required to satisfy the following properties. 
\begin{enumerate}
\item For any $M \in \cC$ the map $\r_{M,1_\cC}$ is the composition 
$$M * 1_\cC \xrightarrow{\sim} M \xrightarrow{\sim} 1_\cC * M$$
of unit isomorphisms (and similarly for $\r_{1_\cC,M}$). 
\item For any $M,N_1,N_2 \in \cC$ we have 
$$\L(M, N_1 * N_2) \leq \L(M, N_1) + \L(M, N_2).$$
If equality holds then 
$$\r_{M,N_1 * N_2} = (\id_{N_1} * \r_{M,N_2}) \circ (\r_{M,N_1} * \id_{N_2})$$
while the right-hand composition is zero if the inequality is strict. The corresponding statement with $M$ on the other side also holds. 

\item For $M,N \in \cC$ we have $\L(M,N) + \L(N,M) \ge 0$. Moreover
$$\r_{N,M} \circ \r_{M,N} \ne 0 \iff \L(M,N) + \L(N,M) = 0.$$
\item For any $M,N_1,N_2 \in \cC$ and morphism $f: N_1 \to N_2$ consider the diagram 
$$\xymatrix{
M*N_1 \ar[r]^{\r_{M,N_1}} \ar[d]_{\id_M * f} & N_1 * M \ar[d]^{f*\id_M} \\
M * N_2 \ar[r]^{\r_{M,N_2}} & N_2 * M 
}$$
where we omit the $\{-\}$ shifts. 
\begin{itemize}
\item If $\L(M,N_1) = \L(M,N_2)$ the diagram commutes.
\item If $\L(M,N_1) < \L(M,N_2)$ the bottom left composition is zero. 
\item If $\L(M,N_1) > \L(M,N_2)$ the top right composition is zero. 
\end{itemize}
The corresponding statements hold when the product with $M$ is taken on the other side. 
\end{enumerate}

\subsection{Deformations}

The key to defining renormalized $\r$-matrices is certain deformations of $\cX_a$ that we now discuss.

Denote by $N_a(z)$ the space of matrices whose eigenvalues are all equal to $z$. We have a corresponding space $M_a(z)$, section $s_a$ and vanishing locus $M_a^{s_a}(z)$ as before. There is still a natural action of $\hG_a$ and, following our earlier conventions, we denote by $\cN_a(z), \cM_a(z)$ and $\cM_a^{s_a}(z)$ their quotients by $\hG_a$. Note that the scaling $\C^\times$ acts trivially on $z$ but the loop one acts with weight $2$ (the same as it acts on $N_a$). 

We need to define a Hall product on $\Coh(\cM_a^{s_a}(z_1)) \times \Coh(\cM_b^{s_b}(z_2))$ for any $z_1,z_2$. We can define the analogue of $\cM_{a,b}^{s_{a,b}}$ as before, which we denote $\cM_{a,b}^{s_{a,b}}(z_1,z_2)$. We have have natural map
$$\pi_{a,b}: \cM_{a,b}^{s_{a,b}}(z_1,z_2) \to \cM_a^{s_a}(z_1) \times \cM_b^{s_b}(z_2).$$

To define $m_{a,b}$ we have to adjust our target category slightly. Denote by $N_{(a,b)}(z_1,z_2)$ the closure of the locus of matrices with (generalized) $z_1$-eigenspaces (resp. $z_2$-eigenspaces) of dimension $a$ (resp. $b$).  We define $M_{(a,b)}^{s_{a+b}}(z_1,z_2)$ and its $\hG_{a+b}$-quotient $\cM_{(a,b)}^{s_{a+b}}(z_1,z_2)$ as before. It is not hard to see that if $z_1=z_2$ then this recovers $\cM_{a+b}^{s_{a+b}}(z)$. Then our old construction still works to define
$$m_{a,b}: \cM_{a,b}^{s_{a,b}}(z_1,z_2) \to \cM_{(a,b)}^{s_{a+b}}(z_1,z_2).$$

Since we can identify $N_a(z)$ with $N_a$ via $N_a \ni [\phi] \mapsto [\phi + zI] \in N_a(z)$ we see that $N_a(z)$ is a trivial deformation of $N_a$ over $\bA^1 = \Spec \C[z]$. Likewise $\cM_a^{s_a}(z)$ is a trivial deformation of $\cM_a^{s_a}$ and for $\cF \in \Coh(\cM_a^{s_a})$ we denote by $\tcF \in \Coh(\cM_a^{s_a}(z))$ its natural (trivial) deformation. For $\cF_a \in \Coh(\cM_a^{s_a})$ and $\cF_b \in \Coh(\cM_b^{s_b})$ we define 
$$\tcF_a * \tcF_b := m_{a,b*} \pi_{a,b}^*(\tcF_a \boxtimes \tcF_b) [u_{a,b}] \{-u_{a,b}-v_{a,b}\} \in \Coh(\cM_{(a,b)}^{s_{a+b}}(z_1,z_2))$$
where $u_{a,b},v_{a,b}$ are as before. The argument used in Proposition \ref{prop:associative} still works to show that this convolution is associative.

\subsection{Over $\bA^2 \setminus \Delta$}\label{sec:diagonal}

Although $\cM_{(a,b)}^{s_{a+b}}(z_1,z_2) \to \bA^2$ is not a trivial family its restriction over the diagonal $\Delta \subset \bA^2$ can be identified with the trivial family $\cM_{a+b}^{s_{a+b}} \times \Delta$. In particular, under this identification, we have $(\tcF_a * \tcF_b)|_{(z,z)} \cong \cF_a * \cF_b$. More difficult is understanding restrictions over $\bA^2 \setminus \Delta$. 

The goal of this section is to show that over $\bA^2 \setminus \Delta$ we have a commutative diagram 
\begin{equation}\label{eq:tau}
\xymatrix{
\bA^2 \setminus \Delta \ar[d]^{S} & \cM_a^{s_a}(z_1) \times \cM_b^{s_b}(z_2) \ar[l] \ar[d]^{S} & & \cM_{a,b}^{s_{a,b}}(z_1,z_2) \ar[d]^{\htau} \ar[r]^{m_{a,b}} \ar[ll]_{\pi_{a,b}} & \cM_{(a,b)}^{s_{a+b}}(z_1,z_2) \ar[d]^{\tau} \\
\bA^2 \setminus \Delta & \cM_b^{s_b}(z_2) \times \cM_a^{s_a}(z_1) \ar[l] & & \cM_{b,a}^{s_{b,a}}(z_2,z_1) \ar[r]^{m_{b,a}} \ar[ll]_{\pi_{b,a}} & \cM_{(b,a)}^{s_{a+b}}(z_2,z_1)
}
\end{equation}
where $S$ is the involution exchanging the factors and the right square is Cartesian.  

First we define $\tau$. Almost from definition we have a natural commutative diagram 
$$\xymatrix{
\cN_{(a,b)}(z_1,z_2) \ar[d] \ar[r]^{\tau} & \cN_{(b,a)}(z_2,z_1) \ar[d] \\
\bA^2 \ar[r]^S & \bA^2
}$$
where $S$ exchanges the two factors of $\bA^2$. If $z_1 \ne z_2$ then the bundle $\C^{a+b}$ has a natural decomposition as $V_a \oplus V_b$ where $V_a$ (resp. $V_b$) is the generalized $z_1$-eigenspace (resp. $z_2$-eigenspace). Exchanging these two factors induces an isomorphism
$$\tau: M_{(a,b)}(z_1,z_2) \xrightarrow{\sim} M_{(b,a)}(z_2,z_1).$$ 
Since this map commutes with $s_{a+b}$ we get an induced map $\tau$ as in (\ref{eq:tau}). 

Next we define $\htau$. To do this note that we can identify $\cN_{a,b}(z_1,z_2)$ with the $\hG_{a+b}$-quotient of
\begin{align}
\label{eq:scheme1} N'_{a,b}(z_1,z_2) := \{(\phi,V): & \phi \text{ fixes } V, \dim(V) = a, \\
\notag & (\phi-z_1I)|_{V} \text{ and } (\phi-z_2I)|_{\C^{a+b}/V} \text{ are nilpotent}\}
\end{align}
where $\phi = (\phi_i) \in \oplus_{i \in I} \End(\C^{a_i+b_i})$ and $V = (V_i) \subset \oplus_{i \in I} \C^{a_i+b_i}$. The morphism $m_{a,b}: \cN_{a,b}(z_1,z_2) \to \cN_{(a,b)}(z_1,z_2)$ is then identified with the map which forgets $V$. Over $\bA^2 \setminus \Delta$ this map is an isomorphism since $V$ can be recovered as the generalized $z_1$-eigenspace.

More generally, we can identify $\cM_{a,b}(z_1,z_2)$ with the $\hG_{a+b}$-quotient of the scheme
\begin{align}
\label{eq:scheme2} M'_{a,b}(z_1,z_2) := \{(\phi,x,V): & \phi \text{ and } x \text{ fix }  V,  \dim(V) = a, \\
\notag & (\phi-z_1I)|_{V} \text{ and } (\phi-z_2I)|_{\C^{a+b}/V} \text{ are nilpotent}\}
\end{align}
where $\phi,V$ as in (\ref{eq:scheme1}) and $x = (x_e) \in M_{a,b}$. 

Now consider  the exact sequence of bundles
$$0 \to K_{a,b}(z_1,z_2) \to M_{a,b}(z_1,z_2) \xrightarrow{\alpha_{a,b}} M_a(z_1) \oplus M_b(z_2) \to 0 $$
over $N'_{a,b}(z_1,z_2)$ where $\alpha_{a,b}$ is the restriction map.The maps 
\begin{align*}
 s_{a,b}: M_{a,b}(z_1,z_2) &\to M_{a,b}(z_1,z_2) \\ 
 (s_a,s_b): M_a(z_1) \oplus M_b(z_2) &\to M_a(z_1) \oplus M_b(z_2)
\end{align*}
commute with $\alpha_{a,b}$ so we can restrict to get $s'_{a,b}: K_{a,b}(z_1,z_2) \to K_{a,b}(z_1,z_2)$. We can identify $K_{a,b}(z_1,z_2)$ with the bundle $\oplus_{e \in E} \Hom(V_{t(e)}, (\C^{a+b}/V)_{h(e)})$ and $s'_{a,b}$ with the commutator $[\phi,x]$. Now $(\phi-z_1I)|_V$ and $(\phi-z_2I)|_{\C^{a+b}/V}$ are nilpotent. Over the locus $\{z_1 \ne z_2\}$ a basic linear algebra argument implies that $[\phi,x]=0$ if and only if $x=0$. In other words, $s'_{a,b}$ is invertible when $z_1 \ne z_2$. Thus, by Lemma \ref{lem:cancel}, $\alpha_{a,b}$ induces an isomorphism
\begin{equation}\label{eq:Miso}
M_{a,b}^{\prime s_{a,b}}(z_1,z_2) \xrightarrow{\sim} M_{(a,b)}^{\prime (s_a,s_b)}(z_1,z_2)
\end{equation}
where $M'_{(a,b)}(z_1,z_2)$ is the total bundle $M_a(z_1) \oplus M_b(z_2)$ over $N'_{a,b}(z_1,z_2)$. We also have the analogous isomorphism 
\begin{equation}\label{eq:Miso2}
M_{b,a}^{\prime s_{b,a}}(z_2,z_1) \xrightarrow{\sim} M_{(b,a)}^{\prime (s_b,s_a)}(z_2,z_1)
\end{equation}
But, over the locus where $z_1 \ne z_2$, there is a natural isomorphism
\begin{equation}\label{eq:iso2}
M_{(a,b)}^{\prime (s_a,s_b)}(z_1,z_2) \to M_{(b,a)}^{\prime (s_b,s_a)}(z_2,z_1)
\end{equation}
induced by exchanging $V_a$ and $V_b$ (which has the effect of exchanging $M_a(z_1)$ and $M_b(z_2)$). Together with \eqref{eq:Miso} and \eqref{eq:Miso2} this gives an isomorphism 
$$M_{a,b}^{\prime s_{a,b}}(z_1,z_2) \xrightarrow{\sim} M_{b,a}^{\prime s_{b,a}}(z_2,z_1).$$
Quotienting by $\hG_{a+b}$ gives us $\htau$. 

In the notation above, the map 
$$m_{a,b}: M_{(a,b)}^{\prime (s_a,s_b)}(z_1,z_2) \to M_{(a,b)}^{s_{a+b}}(z_1,z_2)$$
is induced by the inclusion $M_a(z_1) \oplus M_b(z_2) \to M_{a+b}(z_1,z_2)$. Commutativity of the right square in (\ref{eq:tau}) is then evident. Since $\tau$ and $\htau$ are isomorphisms it follows that this right square is also Cartesian. 

\begin{Proposition}\label{prop:middle}
The right square in (\ref{eq:tau}) is Cartesian while the middle square is commutative. 
\end{Proposition}
\begin{proof}
The first claim was proved above so we focus on the second. Following the notation from (\ref{eq:Miso}) consider the scheme 
$$M'_{(a,b)}(z_1,z_2)^{\beta} := \{(\phi,x,V,\beta_1,\beta_2): x \in M_a \oplus M_b, \beta_1: V \xrightarrow{\sim} \C^a, \beta_2: \C^{a+b}/V \xrightarrow{\sim} \C^b \}$$
where $\phi,V$ are as in (\ref{eq:scheme1}). This scheme is a $G_a \times G_b$ torsor over $M_{(a,b)}^{\prime}(z_1,z_2)$ which in turn is the bundle $M_a \oplus M_b$ over $N'_{a,b}(z_1,z_2)$. There is the similar space
$$M'_{(b,a)}(z_2,z_1)^{\beta} := \{(\phi,x,V',\beta'_1,\beta'_2): x \in M_a \oplus M_b, \beta'_1: V' \xrightarrow{\sim} \C^b, \beta'_2: \C^{a+b}/V' \xrightarrow{\sim} \C^a \}$$
which is a $G_b \times G_a$-torsor over $M'_{(b,a)}(z_1,z_2)$. Considering the vanishing loci for $(s_a,s_b)$ gives us the spaces $M^{\prime (s_a,s_b)}_{(a,b)}(z_1,z_2)^{\beta}$ and $M^{\prime (s_b,s_a)}_{(b,a)}(z_2,z_1)^{\beta}$. Isomorphism (\ref{eq:iso2}) then lifts to an isomorphism of torsors
\begin{equation}\label{eq:iso3} 
M_{(a,b)}^{\prime (s_a,s_b)}(z_1,z_2)^\beta \xrightarrow{\sim} M_{(b,a)}^{\prime (s_b,s_a)}(z_2,z_1)^\beta.
\end{equation}

Using these torsors we can now identify the map $\pi_{a,b}$ as follows. We have the map 
\begin{align}
\label{eq:map2} M^{\prime (s_a,s_b)}_{(a,b)}(z_1,z_2)^\beta & \to M^{s_a}_a(z_1) \times M^{s_b}_b(z_2) \\
\notag (\phi,x,V,\beta_1,\beta_2) & \mapsto [(\beta_1 \cdot \phi|_V, \beta_1 \cdot x|_V), (\beta_2 \cdot \phi|_{\C^{a+b}/V}, \beta_2 \cdot x|_{\C^{a+b}/V})]
\end{align}
where $\beta \cdot (-) = \beta \circ (-) \circ \beta^{-1}$ is conjugation. This map is $G_{a+b}$-equivariant with respect to the natural action 
$$g \cdot (\phi,x,V,\beta_1,\beta_2) = (g\phi g^{-1}, gxg^{-1}, gV, \beta_1g^{-1}, \beta_2g^{-1})$$
on $M_{(a,b)}^{\prime (s_a,s_b)}(z_1,z_2)^\beta$ and the trivial action on $M^{s_a}_a(z_1) \times M^{s_b}_b(z_2)$. Thus we get a map 
$$[M^{\prime (s_a,s_b)}_{(a,b)}(z_1,z_2)^\beta/G_{a+b}] \to M^{s_a}_a(z_1) \times M^{s_b}_b(z_2)$$
and taking quotients by $G_a \times G_b$ on both sides recovers $\pi_{a,b}$. The map $\pi_{b,a}$ is defined similarly using
\begin{align}
\label{eq:map3} M^{\prime (s_b,s_a)}_{(b,a)}(z_2,z_1)^\beta & \to M^{s_b}_b(z_2) \times M^{s_a}_a(z_1) \\
\notag (\phi,x,V',\beta'_1,\beta'_2) & \mapsto [(\beta'_1 \cdot \phi|_{V'}, \beta'_1 \cdot x|_{V'}), (\beta'_2 \cdot \phi|_{\C^{a+b}/V'}, \beta'_2 \cdot x|_{\C^{a+b}/V'})]
\end{align}

Composing (\ref{eq:iso3}) with (\ref{eq:map3}) we get 
\begin{align}
\notag M_{(a,b)}^{\prime (s_a,s_b)}(z_1,z_2)^\beta & \to M^{s_b}_b(z_2) \times M^{s_a}_a(z_1) \\
\notag (\phi,x,V,\beta_1,\beta_2) & \mapsto [(\beta_2 \cdot \phi|_{V'}, \beta_2 \cdot x|_{V'}), (\beta_1 \cdot \phi|_{\C^{a+b}/V'} \beta_1 \cdot x|_{\C^{a+b}/V'})]
\end{align}
whose quotient recovers the composition $\pi_{b,a} \circ \htau$. Once we identify $V'$ with $\C^{a+b}/V$ and $\C^{a+b}/V'$ with $V$ this is clearly the same as the composition of (\ref{eq:map2}) with $S$, whose quotient recovers $S \circ \pi_{a,b}$. This proves the commutativity of the middle square in (\ref{eq:tau}). 
\end{proof}

\subsection{Definition of $\r$}\label{sec:r}

Note that the map $\tau$ discussed in Section \ref{sec:diagonal} over $\bA^2 \setminus \Delta$ extends to an isomorphism 
$$\tau: \cM_{(a,b)}^{s_{a+b}}(z_1,z_2) \xrightarrow{\sim} \cM_{(b,a)}^{s_{a+b}}(z_2,z_1)$$
over $\bA^2$.  Now consider $\cF_a \in \Coh(\cM_a^{s_a})$ and $\cF_b \in \Coh(\cM_b^{s_b})$. Using \eqref{eq:tau} we have 
\begin{align*}
j^* \tau^*(\tcF_b * \tcF_a) 
&\cong \tau^* m_{b,a*} \pi_{b,a}^* j^* (\tcF_b \boxtimes \tcF_a) [u_{a,b}] \\
&\cong m_{a,b*} \htau^* \pi_{b,a}^* j^* (\tcF_b \boxtimes \tcF_a) [u_{a,b}] \\
&\cong m_{a,b*} \pi_{a,b}^* S^* j^* (\tcF_b \boxtimes \tcF_a) [u_{a,b}] \\
&\cong m_{a,b*} \pi_{a,b}^* j^* (\tcF_a \boxtimes \tcF_b) [u_{a,b}] \\
&\cong j^* (\tcF_a * \tcF_b)
\end{align*}
where $j: \bA^2 \setminus \Delta \to \bA^2$ denotes the open embedding as well as its obvious base changes. Note that in the first isomorphism we used that $u_{a,b} = u_{b,a}$ and there are no $\{-\}$ shifts because these are trivial over $\bA^2 \setminus \Delta$. 

On the other hand, $\Delta^*(\tcF_a * \tcF_b) \cong \widetilde{\cF_a * \cF_b}$ and 
$$\Delta^* (\tau^*(\tcF_b * \tcF_a)) \cong \Delta^*(\tcF_a * \tcF_b) \cong \widetilde{\cF_a * \cF_b}.$$

Thus, if we denote by $\cG$ (resp. $\cG'$) the restriction of $\tcF_a * \tcF_b$ (resp. $\tau^*(\tcF_b * \tcF_a)$) to the anti-diagonal $\{(z,-z)\} \cong \bA^1 \subset \bA^2$ then one has natural isomorphisms
\begin{enumerate}
\item $\cG|_{\{0\}} \cong \cF_a * \cF_b$,
\item $\cG'|_{\{0\}} \cong \cF_b * \cF_a$ and 
\item $\cG|_V \cong \cG'|_V$ where $V = \bA^1 \setminus \{0\}$.
\end{enumerate}
Then $M := H^0(\Hom(\cG,\cG'))$ is a $\C[z]$-module whose restriction $M|_V$ is isomorphic to $H^0(\Hom(\cG|_V,\cG'|_V)) \cong H^0(\End(\cG|_V))$. We can then uniquely extend the identity morphism to some $R: \cG \to \cG' \{s\}$ (for some unique $s \in \Z$) such that the restriction $R|_{\{0\}}$ is nonzero. We denote this restriction 
$$\r_{\cF_a,\cF_b}: \cF_a * \cF_b \to \cF_b * \cF_a \{\L(\cF_a,\cF_b)\}$$
where $\L(\cF_a,\cF_b) := s$. 

The definition of these maps is essentially the same as that used in \cite[Section 5.2]{CW18}. The main difference is that the categories here do not carry a chiral structure but rather a sufficiently compatible family of deformations (discussed above). Nevertheless, the argument used in \cite[Theorem 5.10]{CW18} can easily be adjusted to show that these maps form a system of renormalized $\r$-matrices. 

\section{Examples}

We will illustrate with some calculations in the case $Q=Q_n$ is a quiver of type A. We index the vertices of $Q_n$ by $I = \{1,\dots,n\}$ and orient the edges $i \to i+1$. 

If $a = (0^{i-1},1,0^{n-i})$ then $\cX_a = \cM_a^{s_a} \cong [\pt/(\GL_1 \times T)]$. We denote by $\f_{i,\ell} \in \Coh(\cX_a)^\heartsuit$ the sheaf corresponding to the character of $\GL_1$ indexed by $\ell \in \Z$.

\subsection{Case $n=1$}\label{sec:examples1}

We have $\cX_a = [N_a/\hG_a]$ and, in particular, $\cX_1 = [\pt/(\C^\times \times T)]$ and $\cX_2 = [N_2/(GL_2 \times T)]$. The nilpotent cone $N_2$ has two $GL_2$-orbits consisting of the open orbi $U$ and its complement which is a point $p$. The irreducible Koszul-perverse sheaves supported on the point are indexed (up to equivariant and Koszul shifts) by irreducible representations of $GL_2$. For a dominant weight $\mu$ of $GL_2$ we denote by $\cP_{2,\mu}$ the corresponding irreducible. 

\begin{Proposition}\label{prop:calc1}
In $\H_{Q_1}$ every $\f_{1,\ell}$ is real and one has short exact sequences
\begin{align}
\label{eq:ses1} 0 \to \cP_{2,(\ell'-\ell-1)\omega_1 + (\ell+1) \omega_2} \{1\} \to \f_{1,\ell} * \f_{1,\ell'+1} \to \f_{1,\ell+1} * \f_{1,\ell'} \{2\} \to 0 & \text{ if } \ell \le \ell' \\
\label{eq:ses2} 0 \to \f_{1,\ell} * \f_{1,\ell'+1} \{-2\} \to \f_{1,\ell+1} * \f_{1,\ell'} \to \cP_{2,(\ell-\ell'-1)\omega_1+(\ell'+1)\omega_2} \{-1\} \to 0 & \text{ if } \ell \ge \ell'
\end{align}
where, by convention, $\cP_{2,a\omega_1+b\omega_2}=0$ if $a < 0$. 
\end{Proposition}
\begin{proof}
One can identify the map $m_{1,1}: \cN_{1,1} \to \cN_2$ as the $GL_2 \times T$-quotient of the map 
$$\tN_2 := \{(\phi,V): 0 \subset V \subset \C^2, \C^2 \xrightarrow{\phi} V \xrightarrow{\phi} 0 \} \to N_2$$
which forgets $V$ (the Springer resolution). In this notation one can identify $\f_{1,\ell} * \f_{1,\ell}$ with 
$$m_{1,1*}(V^\ell \otimes (\C^2/V)^\ell)[1]\{-1\} \cong m_{1,1*}(\O_{\tN_2} \otimes \det(\C^2)^\ell [1]\{-1\} \cong \O_{N_2} \otimes \det(\C^2)^\ell [1]\{-1\}.$$
This sheaf is simple Koszul-perverse which explains why $\f_{1,\ell}$ is real. 

More generally, one can identify $\f_{1,\ell} * \f_{1,\ell'+1}$ with 
$$m_{1,1*}(\O_{\tN_2} \otimes V^\ell \otimes (\C^2/V)^{\ell'+1})[1]\{-1\}$$
and similarly for $\f_{1,\ell'+1} * \f_{1,\ell}$. Now consider the map of line bundles $\phi: (\C^2/V) \to V \{2\}$. This vanishes precisely along the locus $D$ where $\phi=0$ (and $V$ is arbitrary). In particular, we get a short exact sequence 
$$0 \to \O_{\tN_2} \otimes (\C^2/V) \to \O_{\tN_2} \otimes V \{2\} \to \O_D \otimes V \{2\} \to 0.$$
Twisting by the line bundle $V^\ell \otimes (\C^2/V)^{\ell'}$ we get the exact sequence 
$$0 \to \O_{\tN_2} \otimes V^\ell \otimes (\C^2/V)^{\ell'+1} \to \O_{\tN_2} \otimes V^{\ell+1} \otimes (\C^2/V)^{\ell'} \{2\} \to \O_D \otimes V^{\ell+1} \otimes (\C^2/V)^{\ell'} \{2\} \to 0.$$
Shifting by $[1]\{-1\}$ and applying $m_{1,1*}$ we get an exact triangle 
$$\f_{1,\ell} * \f_{1,\ell'+1} \to \f_{1,\ell'+1} * \f_{1,\ell} \{2\} \to m_{1,1*}(\O_D \otimes V^{\ell+1} \otimes (\C^2/V)^{\ell'}) [1]\{1\}.$$
Since $m_{1,1}$ collapses $D \cong \P^1$ to $p \in N_2$ the right hand term is some sheaf supported on $p$. If $\ell \le \ell'$ then this sheaf is $\Sym^{\ell'-\ell-1}(\C^2) \otimes (\Lambda^2(\C^2))^{\ell+1}[1] \{1\}$ which in our notation above is $\cP_{2,(\ell'-\ell-1)\omega_1 + (\ell+1)\omega_2} [1]\{1\}$. This recovers (\ref{eq:ses1}). The proof of (\ref{eq:ses2}) is similar. 
\end{proof}

\begin{Corollary}\label{cor:calc1}
In $\H_{Q_1}$ one has $\L(\f_{1,\ell}, \f_{1,\ell'}) = \max\{2(\ell' - \ell),-2\}$ for any $\ell,\ell' \in \Z$. 
\end{Corollary}
\begin{proof}
Exchanging $\ell$ and $\ell'$ in (\ref{eq:ses1}) gives the short exact sequence
$$0 \to \cP_{2,(\ell-\ell'-1)\omega_1 + (\ell'+1) \omega_2} \{1\} \to \f_{1,\ell'} * \f_{1,\ell+1} \to \f_{1,\ell'+1} * \f_{1,\ell} \{2\} \to 0$$
where $\ell \ge \ell'$. Using also (\ref{eq:ses2}) we form the composition 
$$\f_{1,\ell+1} * \f_{1,\ell'} \to \cP_{2,(\ell-\ell'-1)\omega_1+(\ell'+1)\omega_2} \{-1\} \to \f_{1,\ell'} * \f_{1,\ell+1} \{-2\}$$
which is nonzero since the first map is surjective and the second injective. It follows that, up to rescaling, this composition must be $\r_{\f_{1,\ell+1},\f_{1,\ell'}}$ when $\ell \ge \ell'$.

On the other hand, if $\ell \le \ell'$ we can repeatedly use (\ref{eq:ses1}) and (\ref{eq:ses2}) to get the sequence 
$$\f_{1,\ell} * \f_{1,\ell'+1} \to \f_{1,\ell+1} * \f_{1,\ell'} \{2\} \to \f_{1,\ell+2} * \f_{1,\ell'-1} \{4\} \to \dots \to \f_{1,\ell'+1} * \f_{1,\ell} \{2(\ell'-\ell+1)\}.$$
Up to the midpoint of this sequence the maps are surjective and after the midpoint they are injective. It follows that the composition above is nonzero and thus, up to rescaling, must agree with $\r_{\f_{1,\ell}, \f_{1,\ell'+1}}$ if $\ell' \le \ell$. The result follows. 
\end{proof}

\subsection{Case $n=2$}\label{sec:examples2}
We have $\cM_{(1,1)} \cong [\Hom(V_1,V_2)/(G_{1,1} \times T)]$ where $V_1,V_2$ have rank one and $G_{(1,1)} \cong \C^\times \times \C^\times$ acts by scaling $V_1 \times V_2$. The section $s_{(1,1)}$ is zero in this case because $\phi_1=\phi_2=0$. Thus 
\begin{equation}\label{eq:X11}
\cX_{(1,1)} = \cM_{(1,1)}^{s_{(1,1)}} \cong [\Hom(V_1,V_2) \times \Spec \C[\veps]/\veps^2]/(G_{(1,1)} \times T)].
\end{equation}
Up to twisting by characters of $G_{(1,1)}$, equivariant and Koszul shifts there is a unique irreducible Koszul-perverse coherent sheaf given by the structure sheaf of $\cM_{(1,1)} = \cX_{(1,1)}^{\cl} \subset \cX_{(1,1)}$. We denote by $\f_{[1,2],(\ell,\ell')}$ this structure sheaf twisted by the characters of $G_{(1,1)}$ indexed by $\ell,\ell' \in \Z$. 

\begin{Proposition}\label{prop:calc2}
In $\H_{Q_2}$ one has short exact sequences
\begin{align}
\label{eq:ses3} & 0 \to \f_{[1,2],(\ell,\ell')}\{-1\} \to \f_{1,\ell} * \f_{2,\ell'} \to \f_{[1,2],(\ell+1,\ell'-1)} [1]\la -1 \ra \to 0 \\
\label{eq:ses4} & 0 \to \f_{[1,2],(\ell+1,\ell'-1)} [1] \la -1 \ra \{-1\} \to  \f_{2,\ell'} * \f_{1,\ell} \to \f_{[1,2],(\ell,\ell')} \to 0 
\end{align}
\end{Proposition}
\begin{proof}
Let us write $W = \Hom(V_1,V_2) = M_{(1,1)}$. If $\{0\} \subset W$ denotes the origin then one has a standard short exact sequence 
$$0 \to \O_W \otimes \rho_1 \rho_2^{-1} \la -1 \ra \{1\} \to \O_W \to \O_{\{0\}} \to 0$$
where $\rho_1,\rho_2$ are the two the two characters of $G_{(1,1)}$ acting on $W$ with weights $-1,1$ respectively. The reasons for the $\la -1 \ra \{1\}$ shift is that the scaling and loop $\C^\times$ act on $W$ with weights $1,-1$ respectively. We can rewrite this sequence as an exact triangle 
\begin{equation}\label{eq:exact1}
\O_W \otimes \rho_1^{\ell} \rho_2^{\ell'} \{-1\} \to \O_{\{0\}} \otimes \rho_1^\ell \rho_2^{\ell'} \{-1\} \to \O_W \otimes \rho_1^{\ell+1} \rho_2^{\ell'-1} [1] \la -1 \ra. 
\end{equation}
for any $\ell, \ell' \in \Z$. Note that the left and right terms in (\ref{eq:exact1}) can be identified, by definition, with the left and right terms in (\ref{eq:ses3}). Thus to show (\ref{eq:ses3}) it remains to check that 
$$\f_{1,\ell} * \f_{2,\ell'} \cong \O_{\{0\}} \otimes \rho_1^\ell \rho_2^{\ell'} \{-1\} \in \Coh(\cX_{(1,1)}).$$
To see this note that 
$$M_{(1,0),(0,1)}^{s_{(1,0),(0,1)}} = M_{(1,0),(0,1)} \xrightarrow{m_{(1,0),(0,1)}} M_{(1,1)}^{s_{(1,1)}}$$
can be identified with the composition $\{0\} \to W = M_{(1,1)} \to M_{(1,1)}^{s_{(1,1)}}$. Thus 
$$\f_{1,\ell} * \f_{2,\ell'} \cong m_{(1,0),(0,1)*}(\O_{\{0\}} \otimes \rho_1^\ell \rho_2^{\ell'}) \{-1\} \cong \O_{\{0\}} \otimes \rho_1^\ell \rho_2^{\ell'} \{-1\}$$
where the $\{-1\}$ shift is because $v_{(1,0),(0,1)} = 1$. This proves (\ref{eq:ses3}). 

On the other hand we also have a standard exact triangle 
\begin{equation}\label{eq:exact3}
 \O_W \otimes \rho_1 \rho_2^{-1} [1]\la -1 \ra \{-1\} \to \O_{M_{(1,1)}^{s_{(1,1)}}} \to \O_W 
\end{equation}
in $\Coh(M_{(1,1)}^{s_{(1,1)}})$. To see this recall the Cartesian square
$$\xymatrix{
M_{(1,1)}^{s_{(1,1)}} \ar[rrr] \ar[d] & & & M_{(1,1)} \ar[d] \\
M_{(1,1)} \ar[rrr] & & & M_{(1,1)} \times_{N_{(1,1)}} M_{(1,1)} \{2\}
}$$
where the right and bottom maps are both of the form $v \mapsto (v,0)$. In particular, the excess bundle in this derived intersection is isomorphic to $M_{(1,1)} \{2\} \cong \O_{N_{(1,1)}} \otimes \rho_1^{-1} \rho_2 \la 1 \ra \{1\}$. This explains the bundles and shifts in the right hand term of (\ref{eq:exact3}). We can then rewrite (\ref{eq:exact3}) as an exact triangle
\begin{equation}\label{eq:exact2}
\O_W \otimes \rho_1^{\ell+1} \rho_2^{\ell'-1} [1]\la -1 \ra \{-1\} \to \O_{M_{(1,1)}^{s_{(1,1)}}} \otimes \rho_1^\ell \rho_2^{\ell'} \to \O_W \otimes \rho_1^\ell \rho_2^{\ell'}.
\end{equation}
Like last time, the left and right terms above can be identified with the left and right terms in (\ref{eq:ses4}). Thus to show (\ref{eq:ses4}) it remains to check that 
$$\f_{2,\ell'} * \f_{1,\ell} \cong \O_{\cX_{(1,1)}} \otimes \rho_1^\ell \rho_2^{\ell'} \in \Coh(\cX_{(1,1)}).$$
This is clear since 
$$M_{(0,1),(1,0)}^{s_{(0,1),(1,0)}} \xrightarrow{m_{(0,1),(1,0)}} M_{(1,1)}^{s_{(1,1)}}$$
is the identity map. Note that this time $v_{(0,1),(1,0)}=0$ so there is no additional $\{-\}$ shift. This proves (\ref{eq:ses4}). 
\end{proof}

The following two results are immediate consequences of (\ref{eq:ses3}) and (\ref{eq:ses4}). 

\begin{Corollary}\label{cor:calc2}
In $\H_{Q_2}$ one has $\L(\f_{1,\ell}, \f_{2,\ell'}) = 1 = \L(\f_{2,\ell'}, \f_{1,\ell})$ for any $\ell,\ell' \in \Z$. 
\end{Corollary}

\begin{Corollary}\label{cor:calc3}
In $\H_{Q_2}$ one has the following exact sequences
\begin{align*}
 0 \to \f_{[1,2],(\ell,\ell')} \{-1\} \to \f_{1,\ell} * \f_{2,\ell'} & \xrightarrow{\r} \f_{2,\ell'} * \f_{1,\ell} \{1\} \to \f_{[1,2],(\ell,\ell')} \{1\} \to 0 \\
 0 \to \f_{[1,2],(\ell,\ell')} [1] \la -1 \ra \{-1\} \to \f_{2,\ell'+1} * \f_{1,\ell-1} & \xrightarrow{\r} \f_{1,\ell-1} * \f_{2,\ell'+1} \{1\} \to \f_{[1,2],(\ell,\ell')} [1] \la -1 \ra \{1\} \to 0 \\
\end{align*}
\end{Corollary}

\subsection{General case}\label{sec:examples3}

For $i \le j$ denote by $[i,j] = (0^{i-1},1^{j-i+1},0^{n-j})$. Then 
$$M_{[i,j]}^{s_{[i,j]}} \cong \prod_{i \le m < j} \left( \Hom(V_m, V_{m+1}) \otimes \Spec \C[\veps_m]/\veps_m^2 \right)$$
whose classical part is just $M_{[i,j]} = \prod_{i \le m < j} \Hom(V_m, V_{m+1})$. Moreover, $G_{[i,j]} = \prod_{i \le m \le j} \GL_1$ so that its characters are indexed by $\uell = (\ell_i,\dots,\ell_j) \in \Z^{j-i+1}$. 

We denote by $\f_{[i,j],\uell}$ the structure sheaf of $M_{[i,j]}$ twisted by the character corresponding to $\uell$. As before we think of this as a sheaf on $\cX_{[i,j]} = [M_{[i,j]}^{s_{[i,j]}}/(G_{[i,j]} \times T)]$. These $\f_{[i,j],\uell}$ are (up to loop and Koszul shifts) all the simple, prime sheaves in $\Coh(\cX_a)^\heartsuit$ for $a \in \{0,1\}^n$. 

The following is a straight-forward generalization of Proposition \ref{prop:calc2}. 

\begin{Proposition}\label{prop:calc3}
In $\H_{Q_n}$ and for $i \le j \le k \le n$ one has short exact sequences
\begin{align}
\label{eq:ses5} & 0 \to \f_{[i,k],\uell \sqcup \uell'} \{-1\} \to \f_{[i,j],\uell} * \f_{[j+1,k],\ell'} \to \f_{[i,k],\uell \vee \uell'} [1]\la -1 \ra \to 0 \\
\label{eq:ses6} & 0 \to \f_{[i,k],\uell \vee \uell'} [1] \la -1 \ra \{-1\} \to  \f_{[j+1,k],\uell'} * \f_{[i,j],\uell} \to \f_{[i,k],\uell \sqcup \uell'} \to 0
\end{align}
where $\sqcup$ denotes concatenation and $\uell \vee \uell' := (\ell_i, \dots, \ell_{j-1}, \ell_j+1, \ell'_{j+1}-1, \ell'_{j+2}, \dots, \ell'_k)$. 
\end{Proposition}

\subsection{Grothendieck groups}\label{sec:examples4}

We denote $f_{i,\ell} := [\f_{i,\ell}]$ and $f_{[i,j],\uell} := [\f_{[i,j],\uell}]$ the classes at the level of Grothendieck groups. The shifts $\{1\}$ and $\la 1 \ra$ correspond to multiplication by $q$ and $t$ respectively. We also denote $[a,b]_q := ab - qab$.

The two sequences in Corollary \ref{cor:calc3} imply that 
$$[f_{i,\ell},f_{i+1,\ell'}]_q = -t [f_{i+1,\ell'+1}, f_{i,\ell-1}]_q.$$
One can further rewrite this as 
$$q^{-1} f_{i,\ell+1} f_{i+1,\ell'} - t f_{i,\ell}f_{i+1,\ell'+1} = f_{i+1,\ell'}f_{i,\ell+1} - tq^{-1} f_{i+1,\ell'+1} f_{i,\ell}$$
which, when $t=1$, reduces to relations $\hat{\rm U}3$ of \cite[Sect. 5.5]{FT19A}. 

Next, one can extend Corollary \ref{cor:calc3} to show that $\L(\f_{[i,i+1]},\f_i)=-1$ and $\L(\f_i, \f_{[i,i+1]})=1$. Thus $\f_{[i,i+1]} * \f_i \cong \f_i * \f_{[i,i+1]} \{1\}$ which means $[f_i,f_{[i,i+1]}]_{q^{-1}}=0$. But, from (\ref{eq:ses3}) and (\ref{eq:ses4}), we have
$$[f_i,f_{i+1}]_q = (q^{-1}-q) f_{[i,i+1]}.$$ 
It follows that $[f_i,[f_i,f_{i+1}]_q]_{q^{-1}} = 0$ which is relation $\hat{\rm U}8$ of \cite[Sect. 5.5]{FT19A}.

We also obtain the following result by applying relations (\ref{eq:ses5}) and (\ref{eq:ses6}) repeatedly. 

\begin{Proposition}\label{prop:K1}
In the Grothendieck group of $\H_{Q_n}$ we have 
\begin{equation}\label{eq:[]}
(q^{-1}-q)^{j-i} f_{[i,j],\uell} = [\dots[f_{i,\ell_i}, f_{i+1,\ell_{i+1}}]_q, \dots, f_{j,\ell_j}]_q
\end{equation}
where $\uell = (\ell_i, \dots, \ell_j)$. 
\end{Proposition}

Note that the right side of (\ref{eq:[]}) appears as the definition of $f_{\beta,r}$ from \cite[Eq. 2.12]{Tsy21} where $\beta = \alpha_i + \dots + \alpha_j$ and $r = (\ell_i,\dots,\ell_j)$. Relatedly, it appears as the definition of $F_{i,j+1}(\ur)$ from \cite[Eq. 3.56]{FT19B}. These elements are used as generators for a PBWD basis for $U_q^{<}(L\gl_{n+1})$ and its integral version (cf. \cite[Thm. 2.16]{Tsy21} and \cite[Thm. 3.57]{FT19B}). Thus, in light of (\ref{eq:[]}), the sheaves $\f_{[i,j],\uell}$ provide a categorification of the generators defined in \cite{Tsy21,FT19B}. 

\appendix

\section{Derived kernels}

In this appendix we collect for reference various results about derived kernels that are used in the main body of the paper. 

\begin{Lemma}\label{lem:van0}
Consider a commutative square of bundles 
$$\xymatrix{
V_1 \ar[r]^f \ar[d]^{s_1} & V_2 \ar[d]^{s_2} \\
W_1 \ar[r]^g & W_2 }$$
over some base $X$. Then $f$ induces a map $\bar{f}: V_1^{s_1} \to V_2^{s_2}$. 
\end{Lemma}
\begin{proof}
The following square is commutative 
$$\xymatrix{
V_1^{s_1} \ar[rr]^{f \circ i_0} \ar[d]^{f \circ i_{s_1}} && V_2 \ar[d]^{\Gamma_{s_2}} \\
V_2 \ar[rr]^-{\Gamma_0} && V_2 \times_X W_2}$$
because 
$$\Gamma_{s_2} \circ (f \circ i_0) = (f \times g) \circ \Gamma_{s_1} \circ i_0 = (f \times g) \circ \Gamma_0 \circ i_{s_1} = \Gamma_0 \circ (f \circ i_{s_1}).$$
It follows that one has a canonical morphism $V_1^{s_1} \to V_2 \times_{V_2 \times_X W_2} V_2 = V_2^{s_2}$.

\end{proof}

\begin{Lemma}\label{lem:van1}
Consider morphisms of bundles $V_1 \xrightarrow{f} V_2 \xrightarrow{g} V_3$ over some base $X$. Then one has the following Cartesian squares
\begin{equation}\label{eq:van-square2A}
\xymatrix{
V_1^{f} \ar[r]^{\bar{\id}} \ar[d]^{i_0} & V_1^{g \circ f} \ar[d]^{\Gamma_f \circ i_0} \\
V_1 \ar[r]^-{\Gamma_f} \ar[d]^f & V_1 \times_{V_3} V_2 \ar[d]^{f \times \id} \\ 
V_2 \ar[r]^-{\Delta_g} & V_2 \times_{V_3} V_2}
\hspace{2cm}
\xymatrix{
V_1^{f} \ar[r]^{\bar{\id}} \ar[d]^{i_f} & V_1^{g \circ f} \ar[d]^{\Gamma_0 \circ i_{g \circ f}} \\
V_1 \ar[r]^-{\Gamma_0} \ar[d]^f & V_1 \times_{V_3} V_2 \ar[d]^{f \times \id} \\ 
V_2 \ar[r]^-{\Delta_g} & V_2 \times_{V_3} V_2}
\end{equation}
\begin{equation}\label{eq:van-square2B}
\xymatrix{
V_1^{g \circ f} \ar[r]^{\bar{f}} \ar[d]^{i_0} & V_2^{g} \ar[d]^{i_0} \\
V_1 \ar[r]^f & V_2  }
\hspace{2cm}
\xymatrix{
V_1^{g \circ f} \ar[r]^{\bar{f}} \ar[d]^{i_{g \circ f}} & V_2^{g} \ar[d]^{i_g} \\
V_1 \ar[r]^f & V_2  }
\end{equation}
\end{Lemma}

\begin{proof}
Given maps $X_1,X_2 \to Y \to Z$ one has the following Cartesian squares 
\begin{equation}\label{eq:general-cart}
\xymatrix{
X_1 \times_Y X_2 \ar[r] \ar[d] & X_1 \times_Z X_2 \ar[d] \\
X_1 \ar[r] \ar[d] & X_1 \times_Z Y \ar[d] \\
Y \ar[r]^-{\Delta} & Y \times_Z Y}
\end{equation}
Applying this in the case $X_1=X_2=V_1, Y = V_1 \times_X V_2$ and $Z = V_1 \times_X V_3$ with maps as in the top of \eqref{eq:local3}
\begin{equation}\label{eq:local3}
\xymatrix{
V_1^f \ar[r]^{i_f} \ar[d]^{i_0} & V_1 \ar[d]^{\Gamma_0} & \\
V_1 \ar[r]^-{\Gamma_f} & V_1 \times_X V_2 \ar[r]^{\id \times g} \ar[d] & V_1 \times_X V_3 \ar[d] \\
 & V_2 \ar[r]^g & V_3 }
\end{equation}
gives us the top two Cartesian squares in \eqref{eq:local1}. 
\begin{equation}\label{eq:local1}
\xymatrix{
V_1 \times_{V_1 \times_X V_2} V_1 \ar[rr] \ar[d]^{i_0} & & V_1 \times_{V_1 \times_X V_3} V_1 \ar[d]^{\id \times \Gamma_0} \\
V_1 \ar[rr]^-{\Gamma_{\Gamma_f}} \ar[d]^{\Gamma_f} & & V_1 \times_{V_1 \times_X V_3} (V_1 \times_X V_2) \ar[d]^{f \times \pi_2} \\
V_1 \times_X V_2 \ar[rr]^-{\Delta} \ar[d] & & (V_1 \times_X V_2) \times_{V_1 \times_X V_3} (V_1 \times_X V_2) \ar[d] \\
V_2 \ar[rr]^-{\Delta} & & V_2 \times_{V_3} V_2 }
\end{equation}
The bottom square in \eqref{eq:local1} is Cartesian because the bottom square in \eqref{eq:local3} is Cartesian. 

Identifying the top line in \eqref{eq:local1} with $V_1^f \to V_1^{g \circ f}$ and using that 
$$V_1 \times_{V_1 \times_X V_3} (V_1 \times_X V_2) \cong V_1 \times_{V_3} V_2$$
yields the Cartesian squares on the left of \eqref{eq:van-square2A}. One similarly shows that the squares on the right of \eqref{eq:van-square2A} are Cartesian. 

Next consider the following cube obtained by base changing the front square with respect to the map $f \times \id$. 
\begin{equation}\label{eq:Ncart}
        \begin{tikzpicture}[baseline=(current  bounding  box.center),thick]
                \newcommand*{\ha}{1.8}; \newcommand*{\hb}{1.8}; \newcommand*{\hc}{1.8};
                \newcommand*{\va}{-1.1}; \newcommand*{\vb}{-1.1}; \newcommand*{\vc}{-1.1};
                \node (ab) at (\ha,0) {$V_1^{g \circ f}$};
                \node (ad) at (\ha+\hb+\hc,0) {$V_1$};
                \node (ba) at (0,\va) {$V_2^g$};
                \node (bc) at (\ha+\hb,\va) {$V_2$};
                \node (cb) at (\ha,\va+\vb) {$V_1$};
                \node (cd) at (\ha+\hb+\hc,\va+\vb) {$V_1 \times_X V_3$};
                \node (da) at (0,\va+\vb+\vc) {$V_2$};
                \node (dc) at (\ha+\hb,\va+\vb+\vc) {$V_2 \times_X V_3$};
		
                \draw[->] (ab) to node[above] {$i_{g \circ f}$} (ad);
                \draw[->] (ba) to node[below] {\hspace{1.5cm} $i_g$} (bc);
                \draw[->] (ab) to node[above] {$\bar{f}$} (ba);
                \draw[->] (ab) to node[left] {} (cb);
                \draw[->] (ad) to node[below] {$f$} (bc);
                \draw[->] (ba) to node[above] {\hspace{-.6cm} $i_0$} (da);
                \draw[->] (cb) to node[above] {\hspace{1.5cm} $ $} (cd);
                \draw[->] (cb) to node[above] {$f$} (da);
                \draw[->] (cd) to node[below] {\hspace{1cm} $f \times \id$} (dc);
                \draw[->] (da) to node[above] {} (dc);

                \draw[-,line width=6pt,draw=white] (ba) to node[above] {\hspace{.5cm} $ $} (bc);
                \draw[->] (ba) to node[above,pos=.75] {$ $} (bc);
                \draw[-,line width=6pt,draw=white] (bc) to  (dc);
                \draw[->] (bc) to node[right,pos=.2] {$ $} (dc);
                \draw[->] (ad) to node[right] {$\Gamma_0$} (cd);
        \end{tikzpicture}
\end{equation}
The top and left side squares in \eqref{eq:Ncart} are the two Cartesian squares in \eqref{eq:van-square2B}. 
\end{proof}

\begin{Corollary}\label{cor:van1}
Consider a morphism of bundles $f: V_1 \to V_2$ over some base $X$ and suppose $s_1,s_2$ are endomorphisms of $V_1,V_2$ respectively that commute with $f$. Then we have the following Cartesian diagrams
\begin{equation}\label{eq:van-square3A}
\xymatrix{
V_1^{s_1} \ar[r]^{\bar{\id}} \ar[d]^{i_0} & V_1^{f \circ s_1} \ar[d]^{\Gamma_{s_1} \circ i_0} \\
V_1 \ar[d]^{s_1} \ar[r]^-{\Gamma_{s_1}} & V_1 \times_{V_2} V_1 \ar[d]^{s_1 \times \id} \\
V_1 \ar[r]^-{\Delta_f} & V_1 \times_{V_2} V_1 }
\hspace{2cm}
\xymatrix{
V_1^{s_1} \ar[r]^{\bar{\id}} \ar[d]^{i_0} & V_1^{f \circ s_1} \ar[d]^{\Gamma_0 \circ i_{f \circ s_1}} \\
V_1 \ar[d]^{s_1} \ar[r]^-{\Gamma_0} & V_1 \times_{V_2} V_1 \ar[d]^{s_1 \times \id} \\
V_1 \ar[r]^-{\Delta_f} & V_1 \times_{V_2} V_1 }
\end{equation}
\begin{equation}\label{eq:van-square3B}
\xymatrix{
V_1^{s_2 \circ f} \ar[r]^{\bar{f}} \ar[d]^{i_0} & V_2^{s_2} \ar[d]^{i_0} \\
V_1 \ar[r]^f & V_2 }
\hspace{2cm}
\xymatrix{
V_1^{s_2 \circ f} \ar[r]^{\bar{f}} \ar[d]^{i_{s_2 \circ f}} & V_2^{s_2} \ar[d]^{i_{s_2}} \\
V_1 \ar[r]^f & V_2 }
\end{equation}
\end{Corollary}
\begin{proof}
The squares in \eqref{eq:van-square3A} are obtained from the squares in \eqref{eq:van-square2A} applied to the composition $V_1 \xrightarrow{s_1} V_1 \xrightarrow{f} V_2$. Likewise, the squares in \eqref{eq:van-square3B} are obtained from the squares in \eqref{eq:van-square2B} applied to the composition $V_1 \xrightarrow{f} V_2 \xrightarrow{s_2} V_2$. 
\end{proof}

\begin{Lemma}\label{lem:van2}
Consider a morphism $s: V \to W$ of bundles over $X$ and its base change $h^*(s): h^*(V) \to h^*(W)$ with respect to some map $h: Y \to X$. Then the following squares are Cartesian
\begin{equation}\label{eq:van-square4}
\xymatrix{
h^*(V)^{h^*(s)} \ar[r]^{\bar{h}} \ar[d] & V^s \ar[d] \\
h^*(V) \ar[r]^h & V }
\end{equation}
\end{Lemma}
\begin{proof}
Applying the base change $h: Y \to X$ to \eqref{eq:van-square} we get 
$$\xymatrix{
h^*(V^s) \ar[r]^{h^*(i_s)} \ar[d]^{h^*(i_0)} & h^*(V) \ar[d]^{\Gamma_0} \\
h^*(V) \ar[r]^{\Gamma_{h^*(s)}} & V \times_X W}
$$
from which it follows that $h^*(V^s) \cong h^*(V)^{h^*(s)}$. The square in \eqref{eq:van-square4} is then the following standard Cartesian square.
$$\xymatrix{
h^*(V^s) \ar[r] \ar[d] & V^s \ar[d] \\
h^*(V) \ar[r] & V }$$
\end{proof}

\begin{Corollary}\label{cor:van2}
Consider the following commutative square
$$\xymatrix{
V_1 \ar[r]^f \ar[d] & V_2 \ar[d] \\
X_1 \ar[r]^g & X_2 }
$$
where $V_1$ and $V_2$ are vector bundles over $X_1$ and $X_2$ respectively and the map $f$ is induced by a morphism of bundles $V_1 \to g^*(V_2)$. Suppose further that $s_1,s_2$ are endomorphisms of $V_1,V_2$ respectively that commute with $f$. Then $f$ induces a morphism $\bar{f}: V_1^{s_1} \to V_2^{s_2}$ such that
\begin{enumerate}
\item $\bar{f}_*$ preserves coherence if $V_1 \to g^*(V_2)$ is injective and $g$ is proper, 
\item $\bar{f}^*$ preserves coherence if $V_1 \to g^*(V_2)$ is surjective and $g^*$ preserves coherence.
\end{enumerate}
\end{Corollary}
\begin{proof}
We factor $f$ as $V_1 \xrightarrow{f_1} g^*(V_2) \xrightarrow{f_2} V_2$ which subsequently induces $\bar{f}$ as the following composition 
$$V_1^{s_1} \xrightarrow{\bar{\id}} V_1^{f_1 \circ s_1} = V_1^{g^*(s_2) \circ f_1} \xrightarrow{\bar{f}_1} g^*(V_2)^{g^*(s_2)} \xrightarrow{\bar{f}_2} V_2^{s_2}.$$
Using \eqref{eq:van-square3A}, the fact $\Delta_{f_1}$ is proper and base change we get that $\bar{\id}$ is proper. Thus $\bar{\id}_*$ preserves coherence. Similarly, if $f_1$ is injective then it is proper and by \eqref{eq:van-square3B} and base change we find that $\bar{f_1}$ is also proper. Thus $\bar{f}_{1*}$ also preserve coherence. Moreover, if $g$ is proper then by \eqref{eq:van-square4} and base change we find that $\bar{f_2}$ is proper and thus $\bar{f}_{2*}$ also preserves coherence. It follows that $\bar{f}_*$ preserves coherence which proves (1). 

On the other hand, if $f_1$ is surjective then $\Delta_{f_1}^*$ preserves coherence and using \eqref{eq:van-square3A} and base change $\bar{\id}^*$ preserves coherence. Similarly, since $f_1^*$ preserves coherence, \eqref{eq:van-square3B} and base change implies that $\bar{f}_1^*$ also preserves coherence. Lastly, if $g^*$ preserves coherence then using (\ref{eq:van-square4} and base change we get that $\bar{f}_2^*$ preserves coherence. It follows that $\bar{f}^*$ preserves coherence which proves (2). 
\end{proof}

\begin{Proposition}\label{prop:van}
Consider the fiber product of bundles over some base $X$ shown on the left in \eqref{eq:cartesian} and suppose that $s_i \in \End(V_i)$ are endomorphisms which commute with these maps. Then the induced square on the right of \eqref{eq:cartesian} is also Cartesian. 
\begin{equation}\label{eq:cartesian}
\xymatrix{
V_1 \ar[r]^f \ar[d]^{g} & V_2 \ar[d]^{g'} \\
V_3 \ar[r]^{f'} & V_4 }
\hspace{2cm}
\xymatrix{
V_1^{s_1} \ar[r]^{\bar{f}} \ar[d]^{\bar{g}} & V_2^{s_2} \ar[d]^{\bar{g}'} \\
V_3^{s_3} \ar[r]^{\bar{f}'} & V_4^{s_4} }
\end{equation}
\end{Proposition}
\begin{proof}
Using Lemma \ref{lem:van0} we can factor the right square in \eqref{eq:cartesian} as 
\begin{equation}\label{eq:local2}
\xymatrix{
V_1^{s_1} \ar[r] \ar[d]^{\bar{g}} & V_1^{f \circ s_1} \ar[r] \ar[d]^{\bar{g}} & V_2^{s_2} \ar[d]^{\bar{g}'} \\
V_3^{s_3} \ar[r] & V_3^{f \circ s_3} \ar[r] & V_4^{s_4} }
\end{equation}
Now consider the following commutative cube obtained by base changing the front face. 
\begin{equation}\label{eq:Ncart2}
        \begin{tikzpicture}[baseline=(current  bounding  box.center),thick]
                \newcommand*{\ha}{1.8}; \newcommand*{\hb}{1.8}; \newcommand*{\hc}{1.8};
                \newcommand*{\va}{-1.1}; \newcommand*{\vb}{-1.1}; \newcommand*{\vc}{-1.1};
                \node (ab) at (\ha,0) {$V_1^{s_1}$};
                \node (ad) at (\ha+\hb+\hc,0) {$V_1^{f \circ s_1}$};
                \node (ba) at (0,\va) {$V_3^{s_3}$};
                \node (bc) at (\ha+\hb,\va) {$V_3^{f' \circ s_3}$};
                \node (cb) at (\ha,\va+\vb) {$V_1$};
                \node (cd) at (\ha+\hb+\hc,\va+\vb) {$V_1 \times_{V_2} V_1$};
                \node (da) at (0,\va+\vb+\vc) {$V_3$};
                \node (dc) at (\ha+\hb,\va+\vb+\vc) {$V_3 \times_{V_4} V_3$};

                \draw[->] (ab) to node[above] {$ $} (ad);
                \draw[->] (ba) to node[below] {\hspace{1.5cm} $ $} (bc);
                \draw[->] (ab) to node[above] {$\bar{g}$} (ba);
                \draw[->] (ab) to node[left] {} (cb);
                \draw[->] (ad) to node[below] {$\bar{g}$} (bc);
                \draw[->] (ba) to node[above] {\hspace{-.6cm} $ $} (da);
                \draw[->] (cb) to node[above] {\hspace{1.5cm} $\Delta$} (cd);
                \draw[->] (cb) to node[above] {$g$} (da);
                \draw[->] (cd) to node[below] {\hspace{1cm} $g \times g$} (dc);
                \draw[->] (da) to node[above] {$\Delta$} (dc);

                \draw[-,line width=6pt,draw=white] (ba) to node[above] {\hspace{.5cm} $ $} (bc);
                \draw[->] (ba) to node[above,pos=.75] {$ $} (bc);
                \draw[-,line width=6pt,draw=white] (bc) to  (dc);
                \draw[->] (bc) to node[right,pos=.2] {$ $} (dc);
                \draw[->] (ad) to node[right] {$ $} (cd);
        \end{tikzpicture}
\end{equation}
The front and back faces are Cartesian by the squares in \eqref{eq:van-square3A} and it is a standard exercise to check that the bottom face is Cartesian. It follows that the top square is also Cartesian and thus the left square in \eqref{eq:local2} is Cartesian. 

Similarly, we can consider the following commutative group obtained by base changing the front face.
\begin{equation}\label{eq:Ncart3}
        \begin{tikzpicture}[baseline=(current  bounding  box.center),thick]
                \newcommand*{\ha}{1.8}; \newcommand*{\hb}{1.8}; \newcommand*{\hc}{1.8};
                \newcommand*{\va}{-1.1}; \newcommand*{\vb}{-1.1}; \newcommand*{\vc}{-1.1};
                \node (ab) at (\ha,0) {$V_1^{s_1}$};
                \node (ad) at (\ha+\hb+\hc,0) {$V_2^{s_2}$};
                \node (ba) at (0,\va) {$V_3^{f' \circ s_3}$};
                \node (bc) at (\ha+\hb,\va) {$V_4^{s_4}$};
                \node (cb) at (\ha,\va+\vb) {$V_1$};
                \node (cd) at (\ha+\hb+\hc,\va+\vb) {$V_2$};
                \node (da) at (0,\va+\vb+\vc) {$V_3$};
                \node (dc) at (\ha+\hb,\va+\vb+\vc) {$V_4$};

                \draw[->] (ab) to node[above] {$\bar{f}$} (ad);
                \draw[->] (ba) to node[below] {\hspace{1.5cm} $\bar{f}'$} (bc);
                \draw[->] (ab) to node[above] {$\bar{g}$} (ba);
                \draw[->] (ab) to node[left] {} (cb);
                \draw[->] (ad) to node[below] {$\bar{g}'$} (bc);
                \draw[->] (ba) to node[above] {\hspace{-.6cm} $ $} (da);
                \draw[->] (cb) to node[above] {\hspace{1.5cm} $f$} (cd);
                \draw[->] (cb) to node[above] {$g$} (da);
                \draw[->] (cd) to node[below] {$g'$} (dc);
                \draw[->] (da) to node[above] {$f'$} (dc);

                \draw[-,line width=6pt,draw=white] (ba) to node[above] {\hspace{.5cm} $ $} (bc);
                \draw[->] (ba) to node[above,pos=.75] {$ $} (bc);
                \draw[-,line width=6pt,draw=white] (bc) to  (dc);
                \draw[->] (bc) to node[right,pos=.2] {$ $} (dc);
                \draw[->] (ad) to node[right] {$ $} (cd);
        \end{tikzpicture}
\end{equation}
The front and back faces are Cartesian by the squares in \eqref{eq:van-square3B}. It follows that the top square is also Cartesian and thus the right square in \eqref{eq:local2} is Cartesian. So the total square in \eqref{eq:local2}, namely the right square in \eqref{eq:cartesian}, must be Cartesian. 
\end{proof}

We end with the following Lemma which is used in Section \ref{sec:r-matrices} to construct renormalized $\r$-matrices. 

\begin{Lemma}\label{lem:cancel}
Consider the following commutative diagram of bundles over some base $X$ where the horizontal lines are exact
\begin{equation}
\xymatrix{
0 \ar[r] & V_1 \ar[r]  \ar[d]^{s_1} & V_2 \ar[d]^{s_2} \ar[r]^f & V_3 \ar[r] \ar[d]^{s_3} & 0 \\
0 \ar[r] & W_1 \ar[r] & W_2 \ar[r] & W_3 \ar[r] & 0 
}
\end{equation}
If $s_1$ is an isomorphism then the induced map $\bar{f}: V_2^{s_2} \to V_3^{s_3}$ is an isomorphism. 
\end{Lemma}
\begin{proof}
The map $\bar{f}$ exists by Lemma \ref{lem:van0}. Being an isomorphism is a local property over $X$ so we can assume that both horizontal short exact sequences split. In this case we can rewrite the map $s_2$ as $(s_1,s_3): V_1 \times_X V_3 \to W_1 \times_X W_3$. In particular, 
$$V_2^{s_2} \cong V_1^{s_1} \times_X V_3^{s_3}.$$
The result follows since $V_1^{s_1} \cong X$ because $s_1$ is an isomorphism. 
\end{proof}

\end{document}